\documentclass[12pt,a4paper,reqno]{amsart}
\pdfoutput=1
\usepackage[hidelinks,pdfstartview=FitH]{hyperref} %backref=page
\usepackage{fullpage}
\usepackage{amssymb,enumitem,mathrsfs}
\usepackage[T1]{fontenc}
\usepackage[english]{babel}
\usepackage[justification=centering,skip=0pt]{caption}
\setenumerate{label=\textup{(\roman*)},itemsep=2pt,topsep=2pt,leftmargin=2em}
\setitemize{itemsep=2pt,topsep=2pt,leftmargin=2em}
\usepackage{tikz}
\usetikzlibrary{arrows,arrows.meta,calc,cd}

\tikzset{
>=Stealth,
node distance=2.5cm,
main node/.style={circle,inner sep=1pt,fill=gray!20},
freccia/.style={->,shorten >=2pt,shorten <=2pt,semithick},
ciclo/.style={out=130, in=50, loop, distance=2cm},
baseline=(current bounding box.center)}

\usepackage{mathtools}
\mathtoolsset{showonlyrefs}

\allowdisplaybreaks[4]
\newcommand{\doi}[1]{\url{https://doi.org/#1}}
\newcommand{\isbn}[1]{\url{https://isbnsearch.org/isbn/#1}}
\newcommand{\arxiv}[1]{\href{https://arxiv.org/abs/#1}{preprint arXiv:#1}}
\newcommand{\web}[1]{\url{#1}}

\renewcommand{\emptyset}{\varnothing}
\numberwithin{equation}{section}

\newtheorem{prop}{Proposition}[section]
\newtheorem{thm}[prop]{Theorem}
\newtheorem{lemma}[prop]{Lemma}
\newtheorem{cor}[prop]{Corollary}
\newtheorem{df}[prop]{Definition}

\theoremstyle{remark}
\newtheorem{ex}[prop]{Example}
\newtheorem{rem}[prop]{Remark}

\usepackage{etoolbox}
\AtEndEnvironment{rem}{\hfill\ensuremath{\square}}

\newcommand{\C}{\mathbb{C}}
\newcommand{\Z}{\mathbb{Z}}
\newcommand{\N}{\mathbb{N}}
\newcommand{\R}{\mathbb{R}}

\newcommand{\inner}[1]{\left<#1\right>}
\newcommand{\ket}[1]{\left|#1\right>}
\newcommand{\ketbra}[2]{\mbox{\ensuremath{\left|#1\right>\hspace{-3pt}\left<#2\right|}}}

\makeatletter
\def\@tocline#1#2#3#4#5#6#7{\relax
  \ifnum #1>\c@tocdepth % then omit
  \else
    \par \addpenalty\@secpenalty\addvspace{#2}%
    \begingroup \hyphenpenalty\@M
    \@ifempty{#4}{%
      \@tempdima\csname r@tocindent\number#1\endcsname\relax
    }{%
      \@tempdima#4\relax
    }%
    \parindent\z@ \leftskip#3\relax \advance\leftskip\@tempdima\relax
    \rightskip\@pnumwidth plus4em \parfillskip-\@pnumwidth
    #5\leavevmode\hskip-\@tempdima
      \ifcase #1
       \or\or \hskip 1em \or \hskip 2em \else \hskip 3em \fi%
      #6 \hskip 0.5em \nobreak\relax
    \dotfill\hbox to\@pnumwidth{\@tocpagenum{#7}}\par
    \nobreak
    \endgroup
  \fi}
\makeatother

\title{On amplified graph C*-algebras as cores\\[5pt] of Cuntz-Krieger algebras}
\date{}%August 2025}
\author[F.~D'Andrea, S.E.~Zegers]{Francesco D'Andrea and Sophie Emma Zegers}

\address[F.~D'Andrea]{Dipartimento di Matematica e Applicazioni ``R.~Caccioppoli'',
	Universit\`a di Napoli Federico II, and I.N.F.N. Sezione di Napoli, Via Cintia, 80126 Napoli, Italy.}
\email{francesco.dandrea@unina.it}

\address[S.E.~Zegers]{Delft institute of applied mathematics, Delft University of Technology, P.O. Box 5031, 2600 GA Delft, The Netherlands}
\email{s.e.zegers@tudelft.nl, sophieemmazegers@gmail.com}

\begin{document}



\begin{abstract}
Given a finite directed acyclic graph $R$, we construct from it two graphs $E_R$ and $F_R$, one by adding a loop at every vertex of $R$ and one by replacing every arrow of $R$ by countably infinitely many arrows. We show that the graph C*-algebra $C^*(F_R)$ is isomorphic to the AF core of $C^*(E_R)$. Examples include C*-algebras of quantum flag manifolds and quantum teardrops. We discuss in detail the quantum Grassmannian $Gr_q(2,4)$ and use our description as AF core to study its CW-structure.
\end{abstract}

\maketitle

\begin{center}
\begin{minipage}{0.8\textwidth}
\parskip=0pt\small\tableofcontents
\end{minipage}
\end{center}

\bigskip

\section{Introduction}\label{sec:1}
In a series of seminal papers, Hong and Szyma{\'n}ski \cite{HS02,HS03,HS08} gave a description of numerous quantum homogeneous spaces as graph C*-algebras.
Among their examples, the quantum sphere $S^{2n-1}_q$ of Vaksman and Soibelman \cite{VS90} is described by a C*-algebra $C(S^{2n-1}_q)$ which is isomorphic to the graph C$^*$-algebra of either the graph $L_{2n-1}$ in Figure~\ref{fig:graphSnq} or the graph $\widetilde{L}_{2n-1}$ in Figure~\ref{fig:sphereB}. 
\begin{figure}[t]
\begin{tikzpicture}[font=\small]

\clip (-1,1.5) rectangle (14,-3.5);

\node[main node] (1) {};
\node (2) [right of=1] {};
\node (3) [right of=2] {};
\node (4) [right of=3] {};
\node (5) [right of=4] {};
\node (6) [right of=5] {};

\filldraw (1) circle (0.06) node[below left] {$1$};
\filldraw (2) circle (0.06);
\filldraw (3) circle (0.06);
\filldraw (4) circle (0.06);
\filldraw (6) circle (0.06) node[below right,inner sep=2pt] {$n$};

\path[freccia] (1) edge[ciclo] (1);
\path[freccia] (2) edge[ciclo] (2);
\path[freccia] (3) edge[ciclo] (3);
\path[freccia] (4) edge[ciclo] (4);
\path[freccia] (6) edge[ciclo] (6);

\path[freccia]
	(1) edge (2)
	(2) edge (3)
	(3) edge (4);
\path[freccia,dashed] (4) edge (6);
\path[freccia]
	(1) edge[bend right] (3)
	(1) edge[bend right=40] (4)
	(2) edge[bend right] (4);
\path[white]
	(1) edge[bend right=60] coordinate (7) (6)
	(2) edge[bend right=50] coordinate (8) (6)
	(3) edge[bend right=40] coordinate (9) (6);
\path[shorten <=1pt]
	(1) edge[out=-60,in=180] (7)
	(2) edge[out=-50,in=180] (8)
	(3) edge[out=-40,in=180] (9);
\path[->,dashed,shorten >=1pt]
	(7) edge[out=0,in=240] (6)
	(8) edge[out=0,in=230] (6)
	(9) edge[out=0,in=220] (6);
\path[freccia,dashed] (4) edge[bend right,dashed] (6);
\end{tikzpicture}

\caption{The graph $L_{2n-1}$.}
\label{fig:graphSnq}

\vspace{6mm}

\begin{tikzpicture}[font=\small]

\clip (-1,-0.5) rectangle (14,1.4);

      \node[main node] (1) {};
      \node (2) [main node,right of=1] {};
      \node (3) [main node,right of=2] {};
      \node (4) [main node,right of=3] {};
      \node (5) [right of=4] {};
      \node (6) [main node,right of=5] {};

      \filldraw (1) circle (0.06) node[below left] {$1$};
      \filldraw (2) circle (0.06);
      \filldraw (3) circle (0.06);
      \filldraw (4) circle (0.06);
      \filldraw (6) circle (0.06) node[below right] {$n$};

      \path[freccia] (1) edge[ciclo] (1);
      \path[freccia] (2) edge[ciclo] (2);
      \path[freccia] (3) edge[ciclo] (3);
      \path[freccia] (4) edge[ciclo] (4);
      \path[freccia] (6) edge[ciclo] (6);

      \path[freccia] (1) edge (2) (2) edge (3) (3) edge (4);
      \path[freccia,dashed] (4) edge (6);
\end{tikzpicture}

\caption{The graph $\widetilde{L}_{2n-1}$.}
\label{fig:sphereB}
\end{figure}
The quantum sphere $C(S_q^{2n+1})$ comes with a natural action of $U(1)$, the group of unit complex numbers. Moreover any graph C*-algebra has a natural $U(1)$ action known as the gauge action. Since the isomorphism $C(S^{2n-1}_q)\to C^*(L_{2n-1})$ is $U(1)$-equivariant, it induces an isomorphism $C(\C P^{n-1}_q):=C(S^{2n-1}_q)^{U(1)}\to C^*(L_{2n-1})^{U(1)}$ between fixed point subalgebras,
where $\C P^{n-1}_q$ is the quantum projective space associated to $S^{2n-1}_q$. As explained in \cite[Sect.~4.3]{HS02}, the core $C^*(L_{2n-1})^{U(1)}$ is isomorphic to the graph C*-algebra $C^*(F_{n-1})$, where $F_{n-1}$ is the graph in Figure~\ref{fig:graphCPnq}.

\begin{figure}[b]
\begin{tikzpicture}[font=\small]

\clip (-1,0.5) rectangle (13.7,-3.5);

\node[main node] (1) {};
\node (2) [right of=1] {};
\node (3) [right of=2] {};
\node (4) [right of=3] {};
\node (5) [right of=4] {};
\node (6) [right of=5] {};

\filldraw (1) circle (0.06) node[below left] {$1$};
\filldraw (2) circle (0.06);
\filldraw (3) circle (0.06);
\filldraw (4) circle (0.06);
\filldraw (6) circle (0.06) node[below right] {$n$};

\path[freccia]
	(1) edge node[fill=white] {$\infty$} (2)
	(2) edge node[fill=white] {$\infty$} (3)
	(3) edge node[fill=white] {$\infty$} (4);
\path[freccia,dashed] (4) edge (6);
\path[freccia]
	(1) edge[bend right] node[fill=white] {$\infty$} (3)
	(1) edge[bend right=40] node[fill=white] {$\infty$} (4)
	(2) edge[bend right] node[fill=white] {$\infty$} (4);
\path[white]
	(1) edge[bend right=60] coordinate (7) (6)
	(2) edge[bend right=50] coordinate (8) (6)
	(3) edge[bend right=40] coordinate (9) (6);
\path[shorten <=1pt]
	(1) edge[out=-60,in=180] node[fill=white] {$\infty$} (7)
	(2) edge[out=-50,in=180] node[fill=white] {$\infty$} (8)
	(3) edge[out=-40,in=180] node[fill=white] {$\infty$} (9);
\path[->,dashed,shorten >=1pt]
	(7) edge[out=0,in=240] (6)
	(8) edge[out=0,in=230] (6)
	(9) edge[out=0,in=220] (6);
\path[freccia,dashed] (4) edge[bend right,dashed] (6);
\end{tikzpicture}

\caption{The graph $F_{n-1}$.}
\label{fig:graphCPnq}

\vspace{8mm}

\begin{tikzpicture}[font=\small]

      \node[main node] (1) {};
      \node (2) [main node,right of=1] {};
      \node (3) [main node,right of=2] {};
      \node (4) [main node,right of=3] {};
      \node (5) [right of=4] {};
      \node (6) [main node,right of=5] {};

      \filldraw (1) circle (0.06) node[below left] {$1$};
      \filldraw (2) circle (0.06);
      \filldraw (3) circle (0.06);
      \filldraw (4) circle (0.06);
      \filldraw (6) circle (0.06) node[below right] {$n$};

      \path[freccia] (1) edge node[fill=white] {$\infty$} (2) (2) edge node[fill=white] {$\infty$} (3) (3) edge node[fill=white] {$\infty$} (4);
      \path[freccia,dashed] (4) edge (6);
\end{tikzpicture}

\caption{The graph $\widetilde{F}_{n-1}$.}
\label{fig:CPnqB}
\end{figure}

The graph $F_{n-1}$ is an example of an \emph{amplified graph}: whenever there is an arrow between two vertices, there are countably infinitely many of them.
We know from \cite{ERS12} that the C*-algebra of an amplified graph only depends on its transitive closure, meaning that one also has the isomorphism $C^*(F_{n-1})\simeq C^*(\widetilde{F}_{n-1})$, where $\widetilde{F}_{n-1}$ is the graph in Figure~\ref{fig:CPnqB}. It is natural to wonder whether $C^*(\widetilde{F}_{n-1})$ is isomorphic to the core of $C^*(\widetilde{L}_{2n-1})$. The answer is not obvious, since the isomorphism $C^*(L_{2n-1})\to C^*(\widetilde{L}_{2n-1})$ is not $U(1)$-equivariant with respect to the gauge actions (see e.g.~the formulas in \cite[Thm.~3.11]{Dan23}), and so it does not induce an isomorphism of cores.

The situation is summarized in the following picture:
\begin{equation}\label{eq:commdiag}
\begin{tikzcd}
C(S^{2n-1}_q) \arrow[r, "\sim","\text{equiv.}"'] & C^*(L_{2n-1}) \arrow[r, "\sim","\text{not equiv.}"'] & C^*(\widetilde{L}_{2n-1}) \\
C(\C P^{n-1}_q) \arrow[u,phantom, sloped, "\subseteq"]\arrow[r, "\sim"] & C^*(L_{2n-1})^{U(1)} \arrow[u,phantom, sloped, "\subseteq"]\arrow[d, sloped,phantom,"\cong"]\arrow[r,"/" marking] &  C^*(\widetilde{L}_{2n-1})^{U(1)} \arrow[u,phantom, sloped, "\subseteq"]\arrow[d,phantom,"(\text{\rotatebox[origin=c]{270}{$\,\cong$}})"] \\
& C^*(F_{n-1})\arrow[r,phantom,"\cong"] & C^*(\widetilde{F}_{n-1})
\end{tikzcd}
\end{equation}
The isomorphisms in the first line are given by explicit maps, the first one is equivariant and induces an isomorphism between cores, while the second doesn't. The isomorphisms written as $\cong$ in the diagram are not given explicitly.

The set inclusions in \eqref{eq:commdiag} are noncommutative analogues of the principal $U(1)$-bundle $S^{2n+1}\to\C P^n$, and are themselves principal in the sense explained below. The graphs $L_{2n-1}$ and $\widetilde{L}_{2n-1}$ are finite with no sinks, and we argue that the gauge action on a unital graph C*-algebra $C^*(E)$ is principal in the sense of Ellwood \cite{Ell00} if and only if the graph $E$ is finite with no sinks.
This follows from combining a few well-known results. Firstly,
the $U(1)$ action on $C^*(E)$ is principal if and only if its restriction to the Peter-Weyl subalgebra $\mathcal{P}_{U(1)}(C^*(E))$ defines a Hopf-Galois extension \cite{BDCH17}. Next, a $\mathscr{O}(U(1))$-extensions is Hopf-Galois if and only if the $\Z$-grading induced by the $U(1)$ action is a strong grading \cite[Thm.~8.1.7]{Mon93}. Next, $\mathcal{P}_{U(1)}(C^*(E))$ is strongly $\Z$-graded if and only if $C^*(E)$ is strongly $\Z$-graded (in the sense of C*-algebras). And finally, a unital graph C*-algebra $C^*(E)$ is strongly $\Z$-graded if and only if the graph $E$ is finite with no sinks~\cite{Haz13,CHR19,LO22}.

It is worth recalling that every $\mathscr{O}(U(1))$-Galois extension is automatically faithfully flat, which in the noncommutative framework replaces the condition of local triviality of a principal bundle. Indeed, since the Hopf algebra $\mathscr{O}(U(1))$ is cosemisimple, every right $\mathscr{O}(U(1))$-comodule is injective \cite[Pag.~290]{Swe69}, and this together with bijectivity of the antipode implies that every $\mathscr{O}(U(1))$-Galois extension is faithfully flat \cite[Theorem I]{Sch90}.

The above arguments motivate considering the inclusions in \eqref{eq:commdiag}, and more generally any embedding
\[
C^*(E)^{U(1)}\lhook\joinrel\longrightarrow C^*(E)
\]
with $E$ finite and with no sinks, as a kind of non-commutative principal (locally trivial) $U(1)$-bundle.

The aim of this paper is to generalize the above correspondence between amplified graph C*-algebras and cores of Cuntz-Krieger algebras as much as possible.
We will focus on \emph{unital} C*-algebras, which as costumary we interpret as describing \emph{compact} quantum spaces.
Given an amplified graph $F$, we want to construct a graph $E$ such that
$C^*(F)$ is isomorphic to the core of $C^*(E)$. We want $C^*(E)$ to be (unital and) strongly $\Z$-graded for the grading induced by the gauge action, thus $E$ must be finite without sinks. For such a graph, $C^*(E)^{U(1)}$ is an AF algebra, cf.~eq.~(3.9) in \cite{R05}. On the other hand, $C^*(F)$ is AF if and only if $F$ has no directed cycles (cf.~\cite[Thm.~2.4]{KPR98} and \cite[Cor.~2.13]{DT05}).

This leads to our main Theorem \ref{thm:1}:
if $R$ is a finite directed acyclic graph with no multiple edges, $E_R$ is the graph obtained from $R$ by adding a loop to every vertex, and $F_R$ is the graph obtained from $R$ by replacing every arrow of $R$ by countably infinitely many arrows, then
\begin{equation}\label{eq:iso}
C^*(F_R)\cong C^*(E_R)^{U(1)} .
\end{equation}
If $R$ is the graph with vertex set $\{1,\ldots,n\}$ and one arrow from $i$ to $j$ for each $i<j$, then $E_R=L_{2n-1}$ and $F_R=F_{n-1}$ are the graphs in Figures \ref{fig:graphSnq} and \ref{fig:graphCPnq}.
If $R$ is the graph with vertex set $\{1,\ldots,n\}$ and one arrow from $i$ to $i+1$ for each $1\leq i<n$, then $E_R=\widetilde{L}_{2n-1}$ and $F_R=\widetilde{F}_{n-1}$ are the graphs in Figures \ref{fig:sphereB} and \ref{fig:CPnqB}.
In order to prove \eqref{eq:iso} we don't give an explicit formula for the isomorphism, but rather compute the dimension group of the two C*-algebras (cf.~Prop.~\ref{prop:1}). This is a result of independent interest. In particular, we show that the dimension groups of $C(\C P^{n-1})$ and $C(\C P^{n-1}_q)$ are not isomorphic.

A byproduct of our main theorem is the isomorphism 
$C^*(F_{n-1})\cong C^*(L_{2n-1})^{U(1)}$, that was discovered first in \cite{HS02}, as well as $C^*(\widetilde{F}_{n-1})\cong C^*(\widetilde{L}_{2n-1})^{U(1)}$.
Further examples include C*-algebras of quantum teardrops
and of the quantum (2,4)-Grassmannian $Gr_q(2,4)$. For the latter, a description as amplified graph C*-algebras was given in \cite{BBKS22,S22}.
Here we get an alternative description as the AF core of a Cuntz-Krieger algebra, which is used to derive a CW-decomposition of $Gr_q(2,4)$ (cf.~Theorem \ref{thm:qCW}).

The structure of the paper is the following. In Sect.~\ref{sec:2} we collect some preliminary notions and results about graph C*-algebras. In Sect.~\ref{sec:maint}
we prove our main theorem about finite directed acyclic graphs without multiple edges.
In Sect.~\ref{sec:4} we study some examples of quantum flag manifolds, in particular $\C P^{n-1}_q$ and $Gr_q(2,4)$, as well as quantum teardrops. In Sect.~\ref{sec:5} we study the CW-structure of $Gr_q(2, 4)$.

\medskip

\begin{center}
\textsc{Acknowledgements}
\end{center}

\noindent
FD is a member of INdAM-GNSAGA (Istituto Nazionale di Alta Matematica “F. Sev-
eri”) -- Unit\`a di Napoli and of INFN -- Sezione di Napoli.
This research is part of the EU Staff Exchange project 101086394 ``Operator Algebras That One Can See''.
We thank S{\o}ren Eilers, Karen Strung, Raemonn O'Buachalla and Elmar Wagner for discussions.

\section{Preliminaries}\label{sec:2}
\subsection{Graph C*-algebras}
Let us recall the definition and some basic properties of graph C*-algebras from \cite{BPRS00,FLR00,R05}. We adopt the conventions of \cite{BPRS00,FLR00}, i.e.~the roles of source and range maps are exchanged with respect to~\cite{R05}.

A directed graph $E=(E^0,E^1,s,t)$ consists of a set $E^0$ of vertices, a set $E^1$
of edges, and source and target maps $s,t:E^1\to E^0$. If $e\in E^1$, $v:=s(e)$ and $w:=t(e)$ we say that $v$ \emph{emits} $e$ and that
$w$ \emph{receives} $e$.

A vertex $v\in E^0$ is called a \emph{sink} if it has no outgoing edges, i.e.~$s^{-1}(v)=\emptyset$, an \emph{infinite emitter} if $|s^{-1}(v)|=\infty$,
\emph{singular} if it is either a sink or an infinite emitter,
\emph{regular} if it is not singular, i.e.~if $s^{-1}(v)$ is a finite non-empty set.
We denote by $E^0_{\mathrm{reg}}$ the set of regular vertices.
A graph $E$ is called \emph{row-finite} if it has no infinite emitters, and \emph{finite} if both $E^0$ and $E^1$ are finite sets.

A Cuntz-Krieger $E$-family for an arbitrary graph $E$ consists of mutually orthogonal projections $\{P_v:v\in E^0\}$ and partial isometries $\{S_e:e\in E^1\}$ with orthogonal ranges satisfying the following relations \cite{FLR00}
\begin{alignat}{4}
S^*_eS_e &= P_{r(e)} && \forall\;e\in E^1 , \tag{CK1}\label{eq:CK1} \\
S_eS_e^* &\leq P_{s(e)} && \forall\;e\in E^1 , \tag{CK2}\label{eq:CK2} \\
\sum_{e\in s^{-1}(v)}\!\! \!S_eS_e^* &=P_v && \forall\;v\in E^0_{\mathrm{reg}} . \tag{CK3}\label{eq:CK3}
\end{alignat}
The graph C*-algebra $C^*(E)$ of a graph $E$ is defined as the universal C*-algebra generated by a Cuntz-Krieger $E$-family, and is unique up to an isomorphism which sends Cuntz-Krieger generators into Cuntz-Krieger generators, see \cite{BPRS00,FLR00,R05}.
The C*-algebra $C^*(E)$ is unital if and only if $E^0$ is finite, and in this case the unit is given by the element
\[
1:=\sum_{v\in E^0}P_v .
\]

The \emph{gauge action} \mbox{$\gamma:U(1)\to\mathrm{Aut}\,C^*(E)$} is defined on generators by $\gamma_u(S_e)=u\,S_e$ and $\gamma_u(P_v)=P_v$ for all $u\in U(1),e\in E^1,v\in E^0$. The induced $\Z$-grading assigns to each $S_e$ degree $+1$, to $S_e^*$ degree $-1$ and to $P_v$ degree $0$.
The C*-subalgebra of gauge-invariant elements is denoted $C^*(E)^{U(1)}$ and called the \emph{core} of $C^*(E)$.

A subset $H\subseteq E^0$ is called \emph{hereditary} if, for every $e\in E^1$, $s(e)\in H$ implies $r(e)\in H$, and it is called \emph{saturated} if, for every $v\in E^0$,
$r(s^{-1}(v))\subseteq H$ implies $v\in H$.

Let $I_H$ be the closed two-sided *-ideal in $C^*(E)$ generated by $\{P_v:v\in H\}$.
If $H\subseteq E^0$ is an hereditary and saturated subset, then $I_H$ is a gauge-invariant ideal. When $E$ is row-finite, there is a $U(1)$-equivariant isomorphism $C^*(E)/I_H\cong C^*(E\smallsetminus H)$, where $E\smallsetminus H$ is the subgraph of $E$ obtained by removing all vertices in $H$ and all edges with range in $H$ (see Cor.~3.5 in \cite{BHRS02} with $B=\emptyset$). Note that if $E$ is an arbitrary graph then the quotient by $I_H$ is still a graph $C^*$-algebra but the graph is more involved \cite{BHRS02}.

\medskip

In this paper we are interested in amplified graph C*-algebras and Cuntz-Krieger algebras. A graph is called \emph{amplified} if, whenever there is an edge between two vertices, there are infinitely many of them \cite{ERS12}.
A graph C*-algebra $C^*(E)$ is a \emph{Cuntz-Krieger algebra} if $E$ is a finite graph and its adjacency matrix is a $\{0,1\}$-matrix with no zero row or column \cite[Rem.~2.8]{R05} (this means that the graph has at most one arrow between any two vertices, and every vertex emits and receives at least one edge).

\medskip

\subsection{The dimension group}
If $A$ is a unital stably finite C*-algebra, then the triple
\[
\big(K_0(A),K_0(A)^+,[1_A]\big)
\]
is an ordered abelian group with order unit, called the \emph{dimension group} of $A$.
Here $K_0(A)^+$ is the set of classes of projections in $M_\infty(A)$ and $[1_A]$ the class of the unit of the algebra.
As costumary, for $x,y\in K_0(A)$, we write
\[
x\leq y \quad \Longleftrightarrow \quad y-x\in K_0(A)^+ .
\]

Both commutative C*-algebras and AF algebras are stably finite.
In particular, unital AF algebras $A$ are classified by their dimension group \cite{Ell76}.
If $A=C(X)$ is commutative, then the partial order is compatible with the multiplication in $K_0(A)=K^0(X)$, that is
\[
x,y\in K^0(X)^+ \quad\Longrightarrow\quad xy\in K^0(X)^+ .
\]
(This simply follows from the fact that the multiplication is induced by the tensor product of vector bundles.) Thus, $K^0(X)^+$ is a sub-semiring of $K^0(X)$.

\section{The main theorem}\label{sec:maint}

Let $V$ be a set and $R\subseteq V\times V$ a relation. We will think of it as a directed graph, with vertex set $V$ and arrows $(v,w)\in R$ with source $v$ and target $w$. Its transitive closure, denoted by $\overline{R}$, is the graph with the same vertex set and an arrow $(v,w)\in\overline{R}$ whenever there is a directed walk (with positive length) in $R$ from $v$ to $w$.

We denote by $\Delta:=\{(v,v)\mid v\in V\}$ the diagonal of $V\times V$.

In the following, we will assume that $V$ is finite and that $R$ is a \emph{directed acyclic graph} (i.e.\ it has no directed cycles, in particular no loops). This is equivalent to saying that $\overline{R}$ is a strict partial order, i.e. that (thought as a graph) it has no loops:
\[
\overline{R}\cap\Delta= \emptyset
\]
(recall that a transitive relation is asymmetric if and only if it is irreflexive).

For future use, we denote by $\underline{R}$ the \emph{transitive reduction} (or \emph{Hasse diagram}) of $R$. By definition, $(v,w)\in\underline{R}$ if and only if
$(v,w)\in R$ and there is no directed walk from $v$ to $w$ in $R$ of length greater than $1$. The relation $\underline{R}$ is the smallest one on $V$ whose transitive closure is $\overline{R}$.

We now construct two more directed graphs associated to $R$.

The first one $E_R:=(E^0_R,E^1_R)$ is the graph obtained from $R$ by adding the diagonal of $V\times V$, thus
\begin{equation}\label{eq:E}
E_R^0:=V, \qquad\quad E_R^1:=R\cup\Delta .
\end{equation}
This is a finite graph with no sinks, thus $C^*(E_R)^{U(1)}$ is a unital AF algebra.

The second graph $F_R=(F^0_R,F^1_R,s,t)$ is obtained by replacing each arrow of $R$ by countably infinite arrows. That is,
\begin{equation}\label{eq:F}
\begin{gathered}
F^0_R:=V, \qquad \quad
F^1_R:=\big\{(v,n,w)\big|(v,w)\in R,n\in\N\big\}  , \\
s(v,n,w):=v , \qquad \quad t(v,n,w):=w.
\end{gathered}
\end{equation}
Since $R$ has no directed cycles, $F_R$ has no directed cycles, hence $C^*(F_R)$ is a unital AF algebra, too. The graph $F_R$ is called the \emph{amplification} of $R$ in \cite{ERS12}.

\begin{ex}
Let $V=\{1,2,3\}$ and let $R=\{(1,2), (1,3), (2,3)\}$ be the relation depicted below:
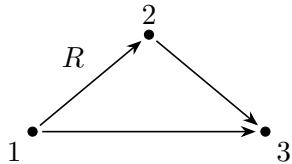
\begin{center}
\begin{tikzpicture}[font=\small]
\node[main node] (1) at (220:2) {};
\node[main node] (2) at (0,0) {};
\node[main node] (3) at (320:2) {};

\node at (-1,-0.3) {$R$};

\filldraw (1) circle (0.06) node[below left] {$1$};
\filldraw (2) circle (0.06) node[above] {$2$};
\filldraw (3) circle (0.06) node[below right] {$3$};
\path[freccia]
	(1) edge (2)
	(2) edge (3)
	(1) edge (3);
\end{tikzpicture} .
\end{center}
The relation $R$, considered as a graph, is indeed acyclic. The graph $E_R$ 
is the graph $L_5$ in Figure \ref{fig:graphSnq}, describing a $5$-dimensional quantum sphere, and $F_R$ is the graph $F_3$ in Figure \ref{fig:graphCPnq} describing the quantum space $\C P^3_q$. 

Note that if, instead, we consider the relation $R'=\{(1,2),(2,3),(3,1)\}$ depicted below
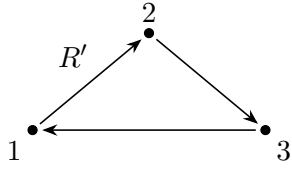
\begin{center}
\begin{tikzpicture}[font=\small]
\node[main node] (1) at (220:2) {};
\node[main node] (2) at (0,0) {};
\node[main node] (3) at (320:2) {};

\node at (-1,-0.3) {$R'$};

\filldraw (1) circle (0.06) node[below left] {$1$};
\filldraw (2) circle (0.06) node[above] {$2$};
\filldraw (3) circle (0.06) node[below right] {$3$};
\path[freccia]
	(1) edge (2)
	(2) edge (3)
	(3) edge (1);
\end{tikzpicture} ,
\end{center}
then $R'$ as a graph is not acyclic, hence it does not fit into our framework. 
\end{ex}

The aim of this section is to prove the following theorem.

\begin{thm}\label{thm:1}
Let $R\subseteq V\times V$ be a finite directed acyclic graph. Then
\[
C^*(E_R)^{U(1)}\cong C^*(F_R) ,
\]
where the graphs $E_R$ and $F_R$ are defined in \eqref{eq:E} and \eqref{eq:F}, respectively.
\end{thm}

We will prove this theorem by computing the dimension group.
For $v,w\in V$, we write $v\preceq w$ if either $v=w$ or $(v,w)\in\overline{R}$.

With an abuse of notations we will denote by  $\{P_v,S_e\}$ the Cuntz-Krieger generators of both $C^*(E_R)$ and $C^*(F_R)$ (it will be clear from the context which one we are referring to).

\begin{prop}\label{prop:1}
Let $A$ be either $C^*(E_R)^{U(1)}$ or $C^*(F_R)$. Then:
\begin{enumerate}
\item\label{en:first} $K_0(A)$ is a free $\Z$-module with basis $\{[P_v]\mid v\in V\}$.

\item\label{en:second} For all $v,w\in V$ and $k\in\Z$,
\[
v \prec w \quad\Longrightarrow \quad [P_v] \geq k[P_w] \; .
\]

%\item\label{en:third} for all $v\in V$, $[P_v]$ is an upper bound for $\mathrm{Span}\big\{[P_w]\;\big|\;v\prec w\big\}$;

\item\label{en:fourth} $[1_A]=\sum_{v\in E^0}[P_v]$.
\end{enumerate}
\end{prop}

\begin{proof}
First note that \ref{en:fourth} is trivial. We now prove the first two points.
Since $k[P_w]\geq (k-1)[P_w]$ for all $w\in V$ and $k\in\Z$, it is enough to prove \ref{en:second} for $k\geq 1$. Also, by transitivity it is enough to prove \ref{en:second} when $(v,w)\in R$.

\medskip
Let us start with the case $A=C^*(F_R)$. 
Since $A$ is an amplified graph we can obtain (i) directly from \cite[Lemma 6.2]{BHRS02} which covers the more general case. We do however include a proof for our case since it simplifies significantly. For any unital graph C*-algebra $C^*(G)$, $K_0$ is the quotient of the free group generated by vertex projections modulo the relations
\[
[P_v]=\sum_{e\in s^{-1}(v)}[P_{r(e)}] \qquad\forall\;v\in G^0_{\mathrm{reg}},
\]
where $G^*_{\mathrm{reg}}$ is the set of regular vertices.
If $v,w\in G^0$ are connected by (at least) $k$ arrows $e_1,\ldots,e_k$, then
\[
[P_v] \geq \sum_{i=1}^k[S_{e_i}S_{e_i}^*]=\sum_{i=1}^k[S_{e_i}^*S_{e_i}]=k[P_w] .
\]
In particular, if $G=F_R$ is an amplified graph, since there are no regular vertices,
$K_0$ is freely generated by the vertex projections and $[P_v] \geq k[P_w]$ for all $(v,w)\in R$ and $k\geq 1$.

\medskip

Now we pass to the case $A=C^*(E_R)^{U(1)}$.
Put $n:=|V|$, let $\Gamma=(\Gamma_{v,w})$ be the adjacency matrix of $E_R$ and $\widetilde{\Gamma}$ be the adjacency matrix of $R$. Note that
$\Gamma=I+\widetilde{\Gamma}$ where $I$ is the identity matrix.
Recall that $(\widetilde{\Gamma}^k)_{v,w}$ is the number of directed walks in $R$ of length $k$ from $v$ to $w$. Any directed walk of length $k\geq n$ passes twice through at least one vertex, which is impossible since $R$ is a directed acyclic graph. We deduce that $\widetilde{\Gamma}^n=0$, and so $\Gamma$ is invertible with inverse
\begin{equation}\label{eq:inverse}
\Gamma^{-1}=\sum_{k=0}^{n-1}(-\widetilde{\Gamma})^k .
\end{equation}
It follows from Proposition 6.10 in \cite{Dan24} that $K_0(A)$ is a free $\Z$-module with basis $\{[P_v]\mid v\in V\}$. For $k\geq 1$ and $v\in V$, put $Q_{v,k}:=S^k_{(v,v)}(S^*_{(v,v)})^k$ (recall that there is a loop $(v,v)\in E_R^1$ at every vertex $v$).
For all $v\in V$ and $k\geq 1$ one has \cite[Lemma 6.9]{Dan24}:
\[
[Q_{v,k}]=\sum_{w\in V}\Gamma_{v,w}[Q_{w,k+1}] .
\]
Since $P_v=Q_{v,0}$, this implies
\begin{equation}\label{eq:Pv}
[P_v]=\sum_{w\in V}(\Gamma^k)_{v,w}[Q_{w,k}] ,
\end{equation}
for all $v\in V$ and $k\geq 1$. Note that for every $(v,w)\in R$, by concatenating loops at $v$ and at $w$ we can form (at least) $k$ distinct paths in $E_R$ from $v$ to $w$. Hence
\begin{equation}\label{eq:Gammak}
(\Gamma^k)_{v,w}\geq k \qquad\forall\;(v,w)\in R,k\geq 1,
\end{equation}
and
\begin{equation}\label{eq:firstineq}
[P_v]\geq k[Q_{w,k}] \qquad\forall\;(v,w)\in R,k\geq 1.
\end{equation}

By inverting \eqref{eq:Pv}, we get
\[
[Q_{v,k}]=\sum_{w\in V}(\Gamma^{-k})_{v,w}[P_w] .
\]
It follows from \eqref{eq:inverse} that, for all $k\geq 1$, $\Gamma^{-k}$ is a sum of powers of $\widetilde{\Gamma}^{k'}$, with $0\leq k'<n$. It follows that $(\Gamma^{-k})_{v,w}=0$ unless $v\preceq w$. 
Thus, for each $k\geq 1$, $[Q_{v,k}]$ is an integer combination of $\{[P_w] : v\preceq w\}$. Then, for all $w'\succ v$ and all $k\geq 1$, one also has
\[
[Q_{w',k}] \in
\mathrm{Span}\big\{[P_w]\;\big|\;w\succeq w'\big\} 
\subseteq\mathrm{Span}\big\{[P_w]\;\big|\;w\succ v\big\} .
\]
Therefore
\begin{equation}\label{eq:span}
\mathrm{Span}\big\{[Q_{w',k}]\;\big|\;k\geq 1, w'\succ v\big\}\subseteq
\mathrm{Span}\big\{[P_w]\;\big|\;w\succ v\big\} .
\end{equation}

Define the \emph{depth} $d(v)$ of $v\in V$ as the maximum length of a directed walk in $R$ with source $v$ (this is an integer less than $|V|$). Note that, if $(v,w)\in R$, then $d(v)\geq d(w)+1$. We will show that
\begin{equation}\label{eq:enough}
[P_v]\geq k[P_w] \qquad\text{for all }(v,w)\in R,k\geq 1,
\end{equation}
by induction on $d(v)$.

If $d(v)=1$, then $w$ is a sink in $R$ and there is only one directed walk of length $k$ in $E_R$ emitted from $w$ (given by $k$ loops). It follows from \eqref{eq:Pv} that $[P_w]=[Q_{w,k}]$ independently of $k$, and so \eqref{eq:firstineq} reduces to \eqref{eq:enough}.

Next, let $N\geq 1$ and assume that \eqref{eq:enough} holds
for all $k\geq 1$ and all $(v,w)\in R$ with $d(v)\leq N$.
Let $(v',w')\in R$ with $d(v')=N+1$.
By inductive hypothesis $[P_v]\geq k [P_w]$
for all $(v,w)\in R$ such that $w'\preceq v\prec w$, and therefore
for all $v,w\in V$ such that $w'\preceq v\prec w$ (by transitivity).

Let $n_{w'}$ is the cardinality of the set $\{w|w\succ w'\}$. By the discussion above, for all integers $(k_w)_{w\succ w'}$ one has
$n_{w'}k_w[P_w]\leq [P_{w'}]$ and then
\[
\sum_{w\succ w'}k_w[P_w]\leq \frac{1}{n_{w'}}\sum_{w\succ w'}[P_{w'}]=[P_{w'}].
\]
This proves that $[P_{w'}]$ is an upper bound for the set $\mathrm{Span}\big\{[P_w]\;\big|\;w\succ w'\big\}$. Because of \eqref{eq:span}, $[P_{w'}]$ is also an upper bound for the set $\mathrm{Span}\big\{[Q_{w,k}]\;\big|\;k\geq 1,w\succ w'\big\}$.

To conclude the proof, from \eqref{eq:Pv} and \eqref{eq:Gammak} we get
\[
[P_{v'}]\geq [Q_{v',k}]+k[Q_{w',k}]
\]
and
\[
[P_{w'}]=[Q_{w',k}]+\xi ,
\]
where
\[
\xi \in \mathrm{Span}\big\{[Q_{w,k}]\;\big|\;k\geq 1,w\succ w'\big\} .
\]
Thus,
\[
[P_{v'}]-(k-1)[P_{w'}]\geq 
[Q_{v',k}]+k[Q_{w',k}]-k([Q_{w',k}]+\xi)+[P_{w'}]
=[Q_{v',k}]+[P_{w'}]-k\xi .
\]
Since $[P_w']$ is an upper bound for the span discussed above, we have $[P_{w'}]\geq k\xi$. Hence
$[P_{v'}]-(k-1)[P_{w'}]\geq 0$ for all $k\geq 1$, which proves the inductive step.
\end{proof}

%Given a non-empty finite subset $X$ of a poset, we denote by $X_{\min}$ and $X_{\max}$ the collection of all minimal resp.\ maximal elements in $X$. Recall that both $X_{\min}$ and $X_{\max}$ are \emph{antichains}, that is, no two different minimal or maximal elements of $X$ are comparable.

For $A$ either $C^*(E_R)^{U(1)}$ or $C^*(F_R)$, using the bijection
\[
\Z^V\ni (k_v)_{v\in V} \mapsto \sum_{v\in V}k_v[P_v] \in K_0(A) ,
\]
we will identify $K_0(A)$ with $\Z^V$ and each element $\sum_{v\in V}[P_v]$ with the corresponding tuple $(k_v)_{v\in V}$. Under this identification $[1_A]$ becomes the constant tuple made of all $1$'s. We now describe $K_0(A)^+$.

If $x:=(k_v)_{v\in V}\in\Z^V$, we call \emph{support} of $x$ the set 
$\mathrm{Supp}(x):=\{v\in V\mid k_v\neq 0\}$.
Given a non-empty finite subset $X$ of a poset, we denote by $X_{\min}$ the collection of all minimal elements in $X$. Recall that this is an \emph{antichains}, that is, no two different minimal elements of $X$ are comparable.
In particular, $\mathrm{Supp}(x)_{\min}$ is the set of minimal elements in $\mathrm{Supp}(x)$ (for the partial order $\preceq$).

\begin{prop}\label{prop:2}
For $A$ either $C^*(E_R)^{U(1)}$ or $C^*(F_R)$, one has
\begin{equation}\label{eq:rhs}
K_0(A)^+=\big\{ x=(k_v)_{v\in V}\;\big|\;k_v>0\;\forall\;v\in\mathrm{Supp}(x)_{\min} \big\} .
\end{equation}
\end{prop}

\begin{proof}
Call $K_R$ the set on the right hand side of \eqref{eq:rhs}.
Since $[P_v]$ is an upper bound for the set $\mathrm{Span}\big\{[P_w]\;\big|\;w\succ v\big\}$, the inclusion $K_R\subseteq K_0(A)^+$ is obvious by using Proposition~\ref{prop:1}. 
We need to show the opposite inclusion.

Let
\[
x=\sum_{v\in V}k_v[P_v] \in K_0(A) ,
\]
and assume that there exists $v_0\in \mathrm{Supp}(x)_{\min}$ such that
$k_{v_0}<0$. We now show that $x\notin K_0(A)^+$.

Let $H$ be the set of all $v\in V$ such that $v\succeq w$ for some $w\in\mathrm{Supp}(x)_{\min}\smallsetminus\{v_0\}$.  $H$ has the following properties 
\begin{itemize}
\item $v_0\notin H$ since minimal elements are incomparable. 
\item $v_0$ is a minimum for
$\mathrm{Supp}(x)\smallsetminus H$ by construction. 
\item $H$ is an hereditary subset both in $E_R$ and in $F_R$ which follows by construction.
\item $H$ is saturated both in $E_R$ and in $F_R$: in $E_R$ because every vertex emits a loop, and in $F_R$ because there are no regular vertices. 
\end{itemize}

Call $R'$ the restriction of $R$ to $V':=V\smallsetminus H$, $E_{R'}$ and $F_{R'}$ the corresponding graphs in \eqref{eq:E} and \eqref{eq:F}, and let $A'$ be either $C^*(E_{R'})^{U(1)}$ or $C^*(F_{R'})$. In the former case, there is a $U(1)$-equivariant surjective *-homomorphism $C^*(E_{R})\to C^*(E_{R'})\cong C^*(E_R)/I_H$,
where $I_H$ is the gauge-invariant ideal in $C^*(E_{R})$ corresponding to $H$,
and this map induces a (surjective) *-homomorphism $C^*(E_{R})^{U(1)}\to C^*(E_{R'})^{U(1)}$.
In the latter case, there is a (surjective) *-homomorphism $C^*(F_{R})\to C^*(F_{R'})\cong C^*(F_R)/J_H$, where $J_H$ is the gauge-invariant ideal in $C^*(F_{R})$ corresponding to $H$. In both cases, we get a *-homomorphism $\pi:A\to A'$ which sends vertex projections to vertex projections and maps the positive cone of $K_0(A)$ to the positive cone of $K_0(A')$.

Clearly $R'$ is still a directed acyclic graph. Let $y$ be the image of the element $x$ above in $K_0(A')$. Note that $-y\in K_0(A')^+$, since $ \mathrm{Supp}(y)_{\min}=\{v_0\}$ and $-k_{v_0}>0$.
If $(G,G^+)$ is any ordered abelian group,
\[
G^+\cap (-G^+)=\{0\} .
\]
Since $-y\in K_0(A')^+$ and $y\neq 0$, then $y\notin K_0(A')^+$. It follows that $x\notin K_0(A)^+$.
\end{proof}

\begin{proof}[Proof of Theorem \ref{thm:1}]
It follows from Propositions \ref{prop:1} and \ref{prop:2} that $C^*(E_R)^{U(1)}$ and $C^*(F_R)$ have the same dimension group. By the classification of unital AF-algebras \cite{Ell76} we obtain $C^*(E_R)^{U(1)}\cong C^*(F_R)$. 
\end{proof}

\section{Applications}\label{sec:4}

\subsection{The dimension group of $\C P^{n-1}_q$}

Let us go back to the graphs in \eqref{eq:commdiag}.
Let $\overline{R}$ be the standard strict total order on $V:=\{1,\ldots,n\}$ and $R$ its transitive reduction. Thus,
 $E_{\overline{R}}=L_{2n-1}$ and $F_{\overline{R}}=F_{n-1}$ are the graphs in Figures \ref{fig:graphSnq} and \ref{fig:graphCPnq},
 while $E_R=\widetilde{L}_{2n-1}$ and $F_R=\widetilde{F}_{n-1}$ are the graphs in Figures \ref{fig:sphereB} and \ref{fig:CPnqB}. The relation $\preceq$
 coincides is the standard total order on $V$, and from Prop.~\ref{prop:2}
we deduce the following corollary.

\begin{cor}\label{cor:E}
For $A$ either one of the algebras $C^*(L_{2n-1})^{U(1)}$, $C^*(F_{n-1})$,
$C^*(\widetilde{L}_{2n-1})^{U(1)}$ and $C^*(\widetilde{F}_{n-1})$,
$K_0(A)\cong\Z^n$ with basis given by the classes of the vertex projections
$\{[P_1],\ldots,[P_n]\}$, and
$K_0(A)^+$ is the set of all integer combinations
\[
a_1[P_1]+\ldots+a_n[P_n]
\]
such that the first non-zero coefficient from the left (if any) is positive.
In particular, the four C*-algebras above are all isomorphic to $C(\C P^{n-1}_q)$
(where $0\leq q<1$).
\end{cor}

The isomorphism between $C^*(L_{2n-1})^{U(1)}$ and $C^*(F_{n-1})$ was discovered 
in \cite[Sect.~4.3]{HS02}. 

\begin{cor}
For $n\geq 2$, the dimension groups of $\C P^{n-1}$ and $\C P^{n-1}_q$
(with $0\leq q<1$) are not isomorphic.
\end{cor}

\begin{proof}
Recall that $K^0(\C P^{n-1})=\Z[x]/(x^n)$, with $x:=1-[\mathcal{L}_1]$ the Euler class of the Hopf line bundle, and any non-zero polynomial in $K^0(\C P^{n-1})^+$ has the constant term which is strictly positive (cf.~App.~\ref{app:CPn}). Let $f:K_0(C(\C P^{n-1}_q))\to K^0(\C P^{n-1})$ be a homomorphism of dimension groups.
Assume, by contradiction, that $f([P_n])\neq 0$.
Since $[P_n]$ is positive, we deduce that
\[
f([P_n])=a_0+O(t) ,
\]
where $a_0\geq 1$.
Since $1-k[P_n]=[P_1]+\ldots+[P_{n-1}]-(k-1)[P_n]$ is also positive (Cor.~\ref{cor:E}), its image
$f(1-k[P_n])=1-ka_0+O(t)$ must be a positive element in $K^0(\C P^{n-1})$, for all $k\geq 0$. But $1-ka_0$ is negative already for $k=2$ (since $a_0\geq 1$), and we get a contradiction.
Thus, $[P_n]$ is in the kernel of $f$. In particular, $f$ cannot be an isomorphism.
\end{proof}

The geometric reason why the dimension groups of $\C P^{n-1}_q$ and $\C P^{n-1}$ are not isomorphic is that there exists non-trivial ``rank $0$ vector bundles'' on $\C P^{n-1}_q$ (for $0\leq q<1$).

\medskip

\begingroup
For $n=2$, we can give an explicit isomorphism $C^*(F_1)\to C^*(L_3)^{U(1)}$. This is a special case of Prop.~\ref{prop:qlens} in the next section.

\subsection{Quantum lens spaces and teardrops}

Let $r\in\Z_+$, $V=\{0,1,\ldots,r\}$ and $R$ the Hasse diagram of the particular point topology on $V$. Explicitly, $(i,j)\in R$ if and only if $i=0$ and $j\geq 1$ (i.e.~a non-empty subset of $V$ is open if and only if it contains $0$).
Then, $E_R$ is the graph $L_3^{r;1,r}$ in Figure \ref{fig:lens}
and $F_R$ is the graph in Figure \ref{fig:teardrop}.
The graph C*-algebra $C^*(L_3^{r;1,r})$ describes the quantum lens space with weight vector $(1,r)$ (see \cite[Example 2.1]{BS18}). Its AF core is isomorphic to the graph C*-algebra of the graph $F_1^{1,r}$ in Figure \ref{fig:teardrop}, which as proved first in \cite{BS18} describes the quantum weighted projective space (or quantum ``teardrop'') $\mathbb{W}P_q^1(1,r)$ with weight vector $(1,r)$.

\begin{figure}[t]
\begin{tikzpicture}[inner sep=3pt]

\node (1) {};
\node (4) [below of=1] {};
\node (3) [left of=4] {};
\node (2) [left of=3] {};
\node (5) [right of=4] {};
\node (6) [right of=5] {};

\draw (5) node {$\ldots$};
\filldraw (1) circle (0.06);
\filldraw (2) circle (0.06);
\filldraw (3) circle (0.06);
\filldraw (4) circle (0.06);
\filldraw (6) circle (0.06);

\path[freccia] (1) edge[ciclo] (1) (1) edge (2) (1) edge (3) (1) edge (4) (1) edge (6);
\path[freccia,ciclo,rotate=180] (2) edge (2) (3) edge (3) (4) edge (4) (6) edge (6);

\end{tikzpicture}
\caption{The graph $L_3^{r;1,r}$.}
\label{fig:lens}

\vspace{5mm}

\begin{tikzpicture}[inner sep=3pt]

\node (1) {};
\node (4) [below of=1] {};
\node (3) [left of=4] {};
\node (2) [left of=3] {};
\node (5) [right of=4] {};
\node (6) [right of=5] {};

\draw (5) node {$\ldots$};
\filldraw (1) circle (0.06);
\filldraw (2) circle (0.06);
\filldraw (3) circle (0.06);
\filldraw (4) circle (0.06);
\filldraw (6) circle (0.06);

\path[freccia] (1) edge node[fill=white,sloped] {$\infty$} (2) (1) edge node[fill=white,sloped] {$\infty$} (3) (1) edge node[fill=white,sloped] {$\infty$} (4) (1) edge node[fill=white,sloped] {$\infty$} (6);

\end{tikzpicture}

\medskip

\caption{The graph $F_1^{1,r}$ of a quantum teardrop.}
\label{fig:teardrop}
\end{figure}

We can give an explicit isomorphism $C^*(F_1^{1,r})\to C^*(L_3^{r;1,r})^{U(1)}$. This is similar to the map in the proof of \cite[Prop.~3.1]{BS18}, which gives an isomorphism between $C^*(F_1^{1,r})$ and a suitable fixed point subalgebra of the quantum sphere algebra $C(S^3_q)$.

In the following, for $1\leq i\leq r$, we denote by $(e_{i,n})_{n\in\N}$ the infinitely many arrows of $F_1^{1,r}$ from the vertex $0$ to the vertex $i$;
we denote by $f_i$ the arrow of $L_3^{r;1,r}$ from the vertex $0$ to the vertex $i$ and, for $0\leq j\leq r$, we denote by $\ell_j$ the loop at $j$ in the graph $L_3^{r;1,r}$.

\begin{prop}\label{prop:qlens}
In the above notations, an isomorphism
\[
\phi:C^*(F_1^{1,r})\longrightarrow C^*(L_3^{r;1,r})^{U(1)}
\]
is given on generators by
\begin{alignat*}{2}
\phi(P_j) &:=P_j &&  \forall\; 0\leq j\leq r ,
\\
\phi(S_{e_{i,n}}) &:=S_{\ell_0}^n S_{f_i}(S_{\ell_i}^*)^{n+1} \qquad && \forall\;1\leq i\leq r,n\in\N.
\end{alignat*}
\end{prop}

\begin{proof}
Firstly, we prove that $\phi$ is well-defined by checking the Cuntz-Krieger relations. The proof of \ref{eq:CK1} is a simple computation:
\[
\phi(S_{e_{i,n}})^*\phi(S_{e_{i,n}})=S_{\ell_i}^{n+1}(S_{\ell_i}^*)^{n+1}=
(S_{\ell_i}^*S_{\ell_i})^{n+1}=P_i,
\]
where we used the fact that $S_{\ell_i}$ is normal, as a consequence Cuntz-Krieger relations and of $i$ emitting only the arrow $\ell_i$.
The condition \ref{eq:CK2} is empty since the graph $F_1^{1,r}$ has no regular vertices. Concerning \ref{eq:CK3}, we start with:
\[
\phi(S_{e_{i,n}})\phi(S_{e_{i,n}})^* =
S_{\ell_0}^nS_{f_i}S_{f_i}^*(S_{\ell_i}^*)^n
\leq 
S_{\ell_0}^nP_0(S_{\ell_0}^*)^n=
S_{\ell_0}^n(S_{\ell_0}^*)^n ,
\]
where we used $S_fS_f\leq S_{s(f)}$, which holds for every arrow $f$.
By the same argument
\[
S_{\ell_0}^n(S_{\ell_0}^*)^n \leq S_{\ell_0}^{n-1} P_0 (S_{\ell_0}^*)^{n-1}=
S_{\ell_0}^{n-1} P_0 (S_{\ell_0}^*)^{n-1}\leq \ldots\leq S_{\ell_0} S_{\ell_0}^* \leq P_0 ,
\]
hence $\phi(S_{e_{i,n}})\phi(S_{e_{i,n}})^* \leq \phi(P_0)$.

The image of $\phi$ is clearly in $C^*(L_3^{r;1,r})^{U(1)}$.
The injectivity of $\phi$ follows from the observation that, for all $0\leq j\leq r$,
$\phi(P_j)$ is not zero and $F_1^{1,r}$ has no cycles \cite[Cor.~2.12]{DT02}.

It remains to prove surjectivity. The core $C^*(L_3^{r;1,r})^{U(1)}$ is spanned by products $S_\alpha S_\beta^*$, where $\alpha,\beta$ are paths with the same length and the same target \cite{R05}.
That is,
\[
S_{\ell_0}^k(S_{\ell_0}^*)^k
\qquad\text{and}\qquad
S_{\ell_0}^jS_{f_i}S_{\ell_i}^{k+m}(S_{\ell_i}^*)^{j+m}S_{f_i}^*
(S_{\ell_0}^*)^k
\]
for $1\leq i\leq r$, $j,k\in\N$ and $m \geq\max\{ -j,-k\}$.
Call $Z_{i,n}:=S_{\ell_i}^n$ if $n>0$,
$Z_{i,n}:=(S_{\ell_i}^*)^{-n}$ if $n<0$ and
$Z_{i,0}:=P_i$. Then, $C^*(L_3^{r;1,r})^{U(1)}$ is spanned by
\begin{equation}\label{eq:ele}
S_{\ell_0}^k(S_{\ell_0}^*)^k
\qquad\text{and}\qquad
S_{\ell_0}^jS_{f_i}Z_{i,k-j}S_{f_i}^*
(S_{\ell_0}^*)^k
\end{equation}
for $j,k\in\N$. These elements are all in the image of $\phi$. Firstly,
\[
\phi(S_{e_{i,j}})
\phi(S_{e_{i,k}}^*)=
S_{\ell_0}^jS_{f_i}Z_{i,k-j}S_{f_i}^*(S_{\ell_0}^*)^k .
\]
Moreover, remembering that the sum of $S_\alpha S_\alpha^*$ over all paths of length $\alpha$ in a graph with fixed length gives the source projection, we find
\begin{multline*}
\sum_{i=1}^r
\phi(S_{e_{i,n}})\phi(S_{e_{i,n}})^*=
\sum_{i=1}^r
S_{\ell_0}^n S_{f_i}(S_{\ell_i}S_{\ell_i}^*)^{n+1}
S_{f_i}^*
(S_{\ell_0}^*)^n \\
=S_{\ell_0}^n
\Big(
\sum_{i=1}^r
S_{f_i}
S_{f_i}^*
\Big)
(S_{\ell_0}^*)^n=
S_{\ell_0}^n
\big(
P_0-S_{\ell_0}S_{\ell_0}^*
\big)
(S_{\ell_0}^*)^n .
\end{multline*}
This is equal to $S_{\ell_0}^n
(S_{\ell_0}^*)^n-
S_{\ell_0}^{n+1}
(S_{\ell_0}^*)^{n+1}$ if $n>0$, and to $P_0-S_{\ell_0}S_{\ell_0}^*$ if $n=0$.
Thus,
\[
\phi(P_0)-\sum_{n=0}^{k-1}
\sum_{i=1}^r
\phi(S_{e_{i,n}})\phi(S_{e_{i,n}})^*=
S_{\ell_0}^k(S_{\ell_0}^*)^k ,
\]
proving that all the elements \eqref{eq:ele} are in the image of $\phi$.
\end{proof}

For $r=1$, we find as a special case the isomorphism between $C(\C P^1_q)\cong C^*(F_1)$ and $C(S^3_q)^{U(1)}\cong C^*(L_3)^{U(1)}$.
It is not clear how to construct an isomorphism from $C(\C P^n_q)\cong C^*(F_n)$ to $C(S^{2n+1}_q)^{U(1)}\cong C^*(L_{2n+1})^{U(1)}$ explicitly for $n\geq 2$.
What makes things easy for the graph $F_1^{1,r}$ of a quantum teardrop is that there are no paths of length greater than $1$. Every arrow in $F_1^{1,r}$ ends on a sink. But already in the case of $\C P^2_q$, it is not clear where to map the isometries attached to the arrows in $F_2$ whose target is the middle vertex. The construction of an explicit isomorphism $C^*(F_R)\to C^*(E_R)^{U(1)}$ for general $R$ is a topic for future work.

\endgroup

\subsection{Circle bundles over quantum flag manifolds}
Let $G$ be a simply connected compact semisimple Lie group, $S$ a subset of simple roots, and $L_S\subset G$ the associated Levi factor.
For $q\in (0,1)$ there is a standard quantization $C_q(G/L_S)$ of the C*-algebra of continuous functions on the flag manifold $G/L_S$, which is a special case of the construction of C*-algebra of a quantum homogeneous space in \cite{NT12}.
It turns out that $C_q(G/L_S)$ is independent of the value of $q$, and isomorphic to the C*-algebra of an amplified graph which can be explicitly constructed as explained in \cite{BBKS22,S22}.

The construction goes as follows \cite{S22}.
The Dynkin diagram determines the Weyl group $W$, which has an abstract presentation in terms of generators $\{s_1,\ldots,s_r\}$, one for each node of the diagram, and relations:
\begin{align*}
s_i^2 & =1 && \forall\;1\leq i\leq r, \\
(s_is_j)^{2+k+\delta_{k,3}} &=1 && \text{if there are  $k$ edges between the nodes $i$ and $j$, $0\leq k\leq 3$}.
\end{align*}
The second relation for $k=0$ simply tells us that $s_i$ and $s_j$ commute, while for $k=1$ it can be conveniently rewritten as a braid relation:
\[
s_is_js_i =s_js_is_j .
\]
The \emph{length} $\ell(w)$ of $w\in W$ the number of generators appearing in any reduced form of $w$.

For $S$ any subset of simple roots, let $W_S$ be the subgroup generated by $\{s_i\mid i\in S\}$. For each left coset in $W/W_S$ we choose a representative of minimal length, and the note by $W^S$ the set of all these representatives.
We define a relation $R$ on $W^S$ as follows:
$(v,w)\in R$ if and only if $w=s_iv$ and $\ell(w)>\ell(v)$ for some generator $s_i$.

It is obvious that $R$ is a finite directed acyclic graph. In fact, it is a graded graph, where each vertex $w\in W^S$ is graded by its length, and arrows go only in the direction where the length increases.
Called $F_R$ the amplified graph associated to this relation $R$, as defined in Sect.~\ref{sec:maint}, then:

\begin{thm}[\cite{BBKS22,S22}]\label{thm:2}
The following isomorphism holds: $C_q(G/L_S)\cong C^*(F_R)$.
\end{thm}

Our Theorem \ref{thm:1} gives an explicit description of $C_q(G/L_S)$ as the core of a Cuntz-Krieger algebra.

Let us work out the explicit example of the Grassmannian of $2$-planes in $\C^4$.
The Dynkin diagram of $SU(4)$ ($A_3$) is:
\begin{center}
\begin{tikzpicture}[inner sep=2pt,semithick,font=\footnotesize]

\node[draw,circle] (a) at (0,0) {};
\node[draw,circle] (b) at (1,0) {};
\node[draw,circle] (c) at (2,0) {};
\draw (a) node[below=3pt] {$1$};
\draw (b) node[below=3pt] {$2$};
\draw (c) node[below=3pt] {$3$};
\draw (a) -- (b);
\draw (b) -- (c);

\end{tikzpicture}
\end{center}
The Weyl group $W$ is isomorphic to the symmetric group on $4$ elements, and has the following 24 elements:
\begin{gather*}
1 \;,\,
s_1 \;,\,
s_2 \;,\,
s_3 \;,\,
\\
s_1s_2 \;,\,
s_2s_1 \;,\,
s_2s_3 \;,\,
s_3s_1 \;,\,
s_3s_2 \;,\,
\\
s_1s_2s_1 \;,\,
s_1s_2s_3 \;,\,
s_2s_3s_1 \;,\,
s_2s_3s_2 \;,\,
s_3s_1s_2 \;,\,
s_3s_2s_1 \;,\,
\\
s_1s_2s_3s_1 \;,\,
s_1s_2s_3s_2 \;,\,
s_2s_3s_1s_2 \;,\,
s_2s_3s_2s_1 \;,\,
s_3s_1s_2s_1 \;,\,
\\
s_1s_2s_3s_1s_2 \;,\,
s_1s_2s_3s_2s_1 \;,\,
s_2s_3s_1s_2s_1 \;,\,
s_1s_2s_3s_1s_2s_1 .
\end{gather*}
The Grassmannian $Gr(2,4)=SU(4)/S\big(U(2)\times U(2)\big)$ is obtained by choosing the subset $S=\{1,3\}$ of simple roots.
In this case, $W_S=\{1,s_1,s_3,s_1s_3\}$ and $W^S$ has elements
$
1$, $
s_2$, $
s_1s_2$, $
s_3s_2$, $
s_1s_3s_2$, $
s_2s_3s_1s_2
$. The graph $R$ is
\begin{center}
\begin{tikzpicture}[inner sep=1pt,font=\small]

\node[left] (1) at (0,0) {$1$};
\node (2) at (2,0) {$s_2$};
\node (3) at (4,1) {$s_1s_2$};
\node (4) at (4,-1){$s_3s_2$};
\node (5) at (6,0) {$s_1s_3s_2$};
\node[right] (6) at (8,0) {$s_2s_1s_3s_2$};

\path[freccia] (1) edge (2);
\path[freccia] (2) edge (3);
\path[freccia] (2) edge (4);
\path[freccia] (3) edge (5);
\path[freccia] (4) edge (5);
\path[freccia] (5) edge (6);

\end{tikzpicture}
\end{center}
It follows from \cite{BBKS22,S22} and Theorem \ref{thm:1} that:

\begin{cor}
$C_q(Gr(2,4))\cong C^*(\widetilde{F}_{2,4})\cong C^*(\widetilde{L}_{2,4})^{U(1)}$, where $\widetilde{F}_{2,4}$ is the graph in Figure~\ref{fig:gr24} and $\widetilde{L}_{2,4}$ the graph in Figure~\ref{fig:graphM9b}.
\end{cor}

\begin{figure}[t]
\begin{tikzpicture}[inner sep=1pt,font=\small]

\node[main node] (1) at (0,0) {1};
\node[main node] (2) at (2,0) {2};
\node[main node] (3) at (4,1) {3};
\node[main node] (4) at (4,-1){4};
\node[main node] (5) at (6,0) {5};
\node[main node] (6) at (8,0) {6};

\path[freccia] (1) edge[ciclo,rotate=90] (1);
\path[freccia] (2) edge[ciclo] (2);
\path[freccia] (3) edge[ciclo] (3);
\path[freccia] (4) edge[ciclo,rotate=180] (4);
\path[freccia] (5) edge[ciclo] (5);
\path[freccia] (6) edge[ciclo,rotate=-90] (6);

\path[freccia] (1) edge (2);
\path[freccia] (2) edge (3);
\path[freccia] (2) edge (4);
\path[freccia] (3) edge (5);
\path[freccia] (4) edge (5);
\path[freccia] (5) edge (6);

\end{tikzpicture}

\caption{The graph $\widetilde{L}_{2,4}$.}\label{fig:graphM9b}

\vspace{1cm}

\begin{tikzpicture}[inner sep=1pt,font=\small]

\node[main node] (1) at (0,0) {1};
\node[main node] (2) at (2,0) {2};
\node[main node] (3) at (4,1) {3};
\node[main node] (4) at (4,-1){4};
\node[main node] (5) at (6,0) {5};
\node[main node] (6) at (8,0) {6};

\path[freccia] (1) edge node[fill=white] {$\infty$}  (2);
\path[freccia] (2) edge node[fill=white,sloped] {$\infty$}  (3);
\path[freccia] (2) edge node[fill=white,sloped] {$\infty$}  (4);
\path[freccia] (3) edge node[fill=white,sloped] {$\infty$}  (5);
\path[freccia] (4) edge node[fill=white,sloped] {$\infty$}  (5);
\path[freccia] (5) edge node[fill=white] {$\infty$}  (6);

\end{tikzpicture}

\medskip

\caption{The graph $\widetilde{F}_{2,4}$.}\label{fig:gr24}
\end{figure}

Similarly to the case of quantum projective spaces, one can replace $R$ by its transitive closure $\overline{R}$. Since the C*-algebra of an amplified graph only depends on its transitive closure, from Theorem \ref{thm:1} again we deduce that:

\begin{cor}\label{cor:geom}
$C_q(Gr(2,4))\cong C^*(L_{2,4})^{U(1)}$, where $L_{2,4}$ is the graph in Figure~\ref{fig:graphM9}.
\end{cor}

\begin{figure}[b]

\begin{tikzpicture}[inner sep=1pt,font=\small]

\clip (-0.6,-4) rectangle (13.2,1.4);

\node[main node] (1) {1};
\node (2) [main node,right of=1] {2};
\node (3) [main node,right of=2] {3};
\node (4) [main node,right of=3] {4};
\node (5) [main node,right of=4] {5};
\node (6) [main node,right of=5] {6};

\path[freccia] (1) edge[ciclo] (1);
\path[freccia] (2) edge[ciclo] (2);
\path[freccia] (3) edge[ciclo] (3);
\path[freccia] (4) edge[ciclo] (4);
\path[freccia] (5) edge[ciclo] (5);
\path[freccia] (6) edge[ciclo] (6);

\path[freccia] (1) edge (2);
\path[freccia] (1) edge[bend right=20] (3);
\path[freccia] (1) edge[bend right=40] (4);
\path[freccia] (1) edge[bend right=60] (5);
\path[freccia] (1) edge[bend right=80] (6);

\path[freccia] (2) edge (3);
\path[freccia] (2) edge[bend right=20] (4);
\path[freccia] (2) edge[bend right=40] (5);
\path[freccia] (2) edge[bend right=60] (6);

\path[freccia] (3) edge[bend right=20] (5);
\path[freccia] (3) edge[bend right=40] (6);

\path[freccia] (4) edge (5);
\path[freccia] (4) edge[bend right=20] (6);

\path[freccia] (5) edge (6);

\end{tikzpicture}

\smallskip

\caption{The graph $L_{2,4}$.}\label{fig:graphM9}
\end{figure}
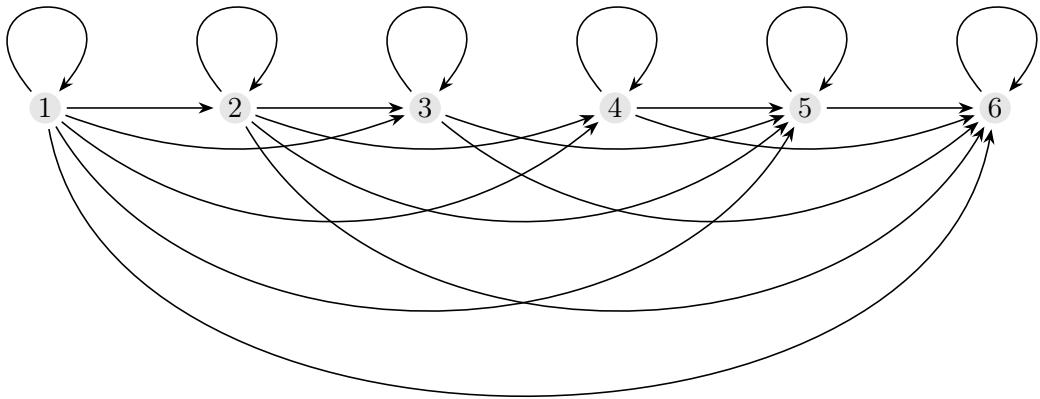

\begin{rem}
In Sect.~\ref{sec:5} we will describe the geometrical interpretation of $C^*(L_{2,4})$, which one can regard as a quantization of the codimension $2$ submanifold of $S^{11}$ defined by Pl{\"u}cker's equation \eqref{eq:pluck}. See Prop.~\ref{prop:M09}, where we study the quantization at $q=0$.
\end{rem}

Finally, there is an isomorphism between $C^*(L_{2,4})$ and $C^*(\widetilde{L}_{2,4})$, which however is not $U(1)$-equivariant and then does not induce an isomorphism of AF cores. This is in analogous to the double presentation of the Vaksman-Soibelman sphere in terms of the graphs $L_{2n-1}$ and $\widetilde{L}_{2n-1}$. Even if there is no general theorem here, we see that in some sense taking the transitive closure of a graph (defined in the appropriate way) does not change the graph C*-algebra. The presentation of the Grassmannian using the graph $L_{2,4}$ will be needed
when studying its CW-structure in Sect.~\ref{sec:5}.

\begin{prop}
There is a non-gauge-equivariant isomorphism $C^*(L_{2,4})\cong C^*(\widetilde{L}_{2,4})$.
\end{prop}
\begin{proof}
Let $E$ and $F$ be directed graphs with finitely many vertices and $A_E, A_F$ their adjacency matrices. In \cite[Cor.~5.6]{AER22} it is shown that $C^*(E)$ and $C^*(F)$ are isomorphic if and only if $E$ can be obtained from $F$ using the moves $(0), (I^+), (R^+), (C^+), (P^+)$ and their inverses.

For a graph $E$ with no sources, let $B_E:=A_E^T-I$. 
WLOG, we assume that $E^0=\{1,\ldots,n\}$.
If two vertices $i,j\in E^0$ are such that $i$ supports a loop and there is a path from $i$ to $j$, then adding row $i$ to row $j$ in $B_E$ corresponds to moves of type $(0)$ and $(R^+)$ \cite[Proposition 4.1]{AER22}. Note that in our case we do not have any regular sources, hence we only have to represent the graph by the matrix $B$ and can avoid the $D$ vector in \cite[Proposition 4.1]{AER22}. Also note that the outsplit move $(0)$ preserves the gauge action but the $(R^+)$ move does not. See \cite[Def.~3.1 and 3.9]{ER19} for the definition of the two moves.

For the two graphs $L_{2,4}$ and $\widetilde{L}_{2,4}$ we have
\[
B_{L_{2,4}}=\begin{bmatrix}
0 & 0 & 0 & 0 & 0 & 0 \\
1 & 0 & 0 & 0 & 0 & 0 \\
1 & 1 & 0 & 0 & 0 & 0 \\
1 & 1 & 0 & 0 & 0 & 0 \\
1 & 1 & 1 & 1 & 0 & 0 \\
1 & 1 & 1 & 1 & 1 & 0
\end{bmatrix}, \qquad\quad
B_{\widetilde{L}_{2,4}}:=\begin{bmatrix}
0 & 0 & 0 & 0 & 0 & 0 \\
1 & 0 & 0 & 0 & 0 & 0 \\
0 & 1 & 0 & 0 & 0 & 0 \\
0 & 1 & 0 & 0 & 0 & 0 \\
0 & 0 & 1 & 1 & 0 & 0 \\
0 & 0 & 0 & 0 & 1 & 0
\end{bmatrix} .
\]
From the discussion above, if we add row $i$ to row $j$ in $B_{L_{2,4}}$
with $i<j$ and $(i,j)\neq (3,4)$ we don't change the isomorphism class of the C*-algebra.
By doing the following row operation, in the given order, we can obtain $B_{L_{2,4}}$ from $B_{\widetilde{L}_{2,4}}$: 
add row $2$ to row $3$ and row $4$,
add row $4$ to row $5$,
add row $5$ to row $6$.
\end{proof}

From now on, we switch notation from $C_q(Gr(2,4))$ to $C(Gr_q(2,4))$ and interpret non-commutative C*-algebras as dual objects of compact quantum spaces. Using this dual geometric point of view, we can interpret the embedding $C(Gr_q(2,4))\to C^*(L_{2,4})$ given by Cor.~\ref{cor:geom} as dual to a noncommutative principal $U(1)$-bundle over $Gr_q(2,4)$.

\section{A CW-type decomposition of $Gr_q(2,4)$}\label{sec:5}

\subsection{The classical case}
Before passing to the quantum case, let us illustrate the geometric picture that we want to ``quantize'', starting with the classical analogue of the extension in Cor.~\ref{cor:geom}. This is the
principal $U(1)$-bundle:
\begin{equation}\label{eq:bundle}
M^9:=\frac{SU(4)}{SU(2)\times SU(2)}\longrightarrow Gr(2,4)\cong\frac{SU(4)}{S\big(U(2)\times U(2)\big)}  .
\end{equation}
where the quotients are w.r.t.~the action by left multiplication. For $i,j\in\{1,\ldots,4\}$, with $i\neq j$, and $U=(u_{ij})\in SU(4)$, let $x_{ij}(U):=u_{1i}u_{2j}-u_{1j}u_{2i}$ be the determinant of the submatrix formed by the elements at the intersection of the first two rows with the columns $i$ and $j$ of $U$.
Then,
\[
(x_{12},x_{13},x_{14},x_{23},x_{24},x_{34}) : SU(4)\to\C^6
\]
has image in the unit sphere $S^{11}$, and induces a diffeomorphism between $M^9$ and the submanifold of $S^{11}$ defined by the Pl{\"u}cker equation:
\begin{equation}\label{eq:pluck}
x_{12}x_{34}-x_{13}x_{24}+x_{14}x_{23}=0 .
\end{equation}

To describe the CW structure of $Gr(2,4)$, it is convenient to think of $Gr(2,4)$ as the quotient of the space of full-rank $2\times 4$ complex matrices by the action by left multiplication by $GL_2(\C)$. The rows of such a matrix form a basis of a 2-dimensional complex vector subspace of $\C^4$, and the $GL_2(\C)$-action changes the basis but not the subspace.
For $A=(a_{ij})\in M_{2\times 4}(\C)$, we denote by $x_{ij}(A):=a_{1i}a_{2j}-a_{1j}a_{2i}$ the same minors as before.
A point in $Gr(2,4)$ with $x_{34}\neq 0$ is represented by a unique matrix in row-echelon form as follows:
\[
\begin{bmatrix}
* & * & 1 & 0 \\
* & * & 0 & 1
\end{bmatrix} .
\]
These points form a subspace of $Gr(2,4)$ homeomorphic to $\C^4\cong\R^8$ (the interior of an $8$-cell).
The complement $X^6$ of this subspace is the $6$-skeleton of the CW-structure.
Every matrix satisfying $x_{34}=0$ can be transformed with the $GL_2(\C)$-action into one of the form:
\[
\begin{bmatrix}
* & * & 0 & 0 \\
* & * & * & *
\end{bmatrix} .
\]
Thus, points in $X^6$ are $2$-dimensional complex vector subspaces of $\C^4$ whose intersection with the subspace \mbox{$\big\{(x,y,0,0):x,y\in\C\big\}$} has dimension at least $1$.
Identifying $\C P^n$ with points in $\C P^{n+1}$ with last homogeneous coordinate equal to $0$, we can view $X^6$ as a quadric in $\C P^4$.
By using the row-echelon form again, we see that each point of the 6-skeleton is represented by a unique matrix in the following list:
\[
\begin{array}{ccccccccc}
\begin{bmatrix}
* & 1 & 0 & 0 \\
* & 0 & * & 1
\end{bmatrix} &&
\begin{bmatrix}
* & 1 & 0 & 0 \\
* & 0 & 1 & 0
\end{bmatrix} &&
\begin{bmatrix}
1 & 0 & 0 & 0 \\
* & * & 1 & 0
\end{bmatrix} &&
\begin{bmatrix}
1 & 0 & 0 & 0 \\
0 & * & 1 & 0
\end{bmatrix} &&
\begin{bmatrix}
1 & 0 & 0 & 0 \\
0 & 1 & 0 & 0
\end{bmatrix} \\
\rule{0pt}{12pt}%
\text{\small (6-cell)} &&
\text{\small (4-cell)} &&
\text{\small (4-cell)} &&
\text{\small (2-cell)} &&
\text{\small (a point)}
\end{array}
\]
We see that there is one 6-cell, two 4-cells, one 2-cell and one point (the $0$-skeleton).
The $2$-cell together with the $0$-skeleton forms the subset of $\C P^5$ of equation $x_{14}=x_{23}=x_{24}=x_{34}=0$, which is $\C P^1$. Together with the two $4$-cells they form the subset of $\C P^5$ of equation $x_{24}=x_{34}=0$ and $x_{14}x_{23}=0$ (coming from Pl{\"u}cker equation), which means that either $x_{14}=0$ or $x_{23}=0$. If one of the two coordinates $x_{14}$ or $x_{23}$ is zero, we get a submanifold diffeomorphic to $\C P^2$. Their intersection is the $2$-skeleton $\C P^1$.
The filtration by skeleta is then
\begin{equation}\label{eq:classicalCW}
\{\text{pt.}\}\longrightarrow\C P^1\longrightarrow\C P^2\sqcup_{\C P^1}\C P^2\longrightarrow 
X^6\longrightarrow Gr(2,4) .
\end{equation}
This is the picture that we want to generalize to the quantum case, cf.~\eqref{eq:qCW}. In the dual language of algebras, of course all the embeddings of spaces become surjective \mbox{*-homomorphisms} of C*-algebras. The fact that each skeleton $X^k$ is obtained from previous skeleton $X^{k-1}$ by attaching $k$-cells becomes the condition that one has a pullback diagram of C*-algebras of the form
\[
\begin{tikzpicture}[scale=1.7]

\node (A) at (1,2) {$C(X^k)$};
\node (B) at (0,1) {$C(X^{k-1})$};
\node (C) at (2,1) {$\coprod C(B^k)$};
\node (D) at (1,0) {$\coprod C(S^{k-1})$};

\path[<-]
     	(B) edge (A)
		(C) edge (A)
		(D) edge (B)
		(D) edge[font=\footnotesize,inner sep=2pt] node[below right]{$\partial$} (C);

\end{tikzpicture} ,
\]
where $B^k$ is the $k$-dimensional closed unit ball, and each sphere $S^{k-1}$ is mapped to a ball $B^k$ via the natural boundary map, whose pullback is a surjective *-homomorphism $\partial:C(B^k)\to C(S^{k-1})$. In the noncommutative setting, in a CW-structure we allow $\partial$ to be replaced by a quotient map of C*-algebras that we can regard as a noncommutative version of the (dual of) a boundary map from a sphere to a ball (one can see \cite{DHMSZ20} for the precise definition). A minimal requirement for such a map is that it is between C*-algebras with the same K-theory as $B^k$ and $S^{k-1}$. We will see that for $Gr_q(2,4)$ this condition is satisfied: what we derive is a strict CW-structure in the sense of \cite{DHMSZ20}, see Theorem \ref{thm:qCW}.

\subsection{Some auxiliary quantum spaces}

From now on, we are interested in C*-algebras up to an isomorphism.
Recall that the graph of $C(S^{2n+1}_q)$ has $n+1$ vertices, here labelled $1,\ldots,n+1$, and one edge from $i$ to $j$ for all $i\leq j$ (Figure \ref{fig:graphSnq}). The fixed-point C*-subalgebra is $C(S^{2n+1}_q)^{U(1)}=C(\C P^n)$.
Removing the loop at the vertex $n+1$ from the graph in Figure \ref{fig:graphSnq} we get the graph C*-algebra $C(B^{2n}_q)$ describing a $2n$-dimensional closed quantum ball \cite{HS08}.

Let us now introduce the quantum spaces that will be the skeleta of the CW-structure of $Gr_q(2,4)$. These are listed in Table \ref{tab:1}. One can see the analogy with \eqref{eq:classicalCW}.

\begin{table}[h]
\begin{center}
\begin{tabular}{|l|c|}
\hline
\multicolumn{1}{|c|}{\bf Amplified graph} & \bf Compact quantum space \\
\hline
\hline
\begin{tikzpicture}[baseline={([yshift=-3pt]current bounding box.center)},font=\footnotesize,scale=0.7]

\node[main node] (1) at (0,0) {};
\node[main node] (2) at (2,0) {};
\node[main node] (3) at (4,1) {};
\node[main node] (4) at (4,-1){};
\node[main node] (5) at (6,0) {};
\node[main node] (6) at (8,0) {};
\node at (0,-1.2) {};
\node at (0,1.2) {};

\filldraw (1) circle (0.07);
\filldraw (2) circle (0.07);
\filldraw (3) circle (0.07);
\filldraw (4) circle (0.07);
\filldraw (5) circle (0.07);
\filldraw (6) circle (0.07);

\path[freccia] (1) edge node[fill=white] {$\infty$}  (2);
\path[freccia] (2) edge node[fill=white,sloped] {$\infty$}  (3);
\path[freccia] (2) edge node[fill=white,sloped] {$\infty$}  (4);
\path[freccia] (3) edge node[fill=white,sloped] {$\infty$}  (5);
\path[freccia] (4) edge node[fill=white,sloped] {$\infty$}  (5);
\path[freccia] (5) edge node[fill=white] {$\infty$}  (6);

\end{tikzpicture}
& $Gr_q(2,4)$ \\ \hline
\begin{tikzpicture}[baseline={([yshift=-3pt]current bounding box.center)},font=\footnotesize,scale=0.7]

\node[main node] (1) at (0,0) {};
\node[main node] (2) at (2,0) {};
\node[main node] (3) at (4,1) {};
\node[main node] (4) at (4,-1){};
\node[main node] (5) at (6,0) {};
\node at (0,-1.2) {};
\node at (0,1.2) {};

\filldraw (1) circle (0.07);
\filldraw (2) circle (0.07);
\filldraw (3) circle (0.07);
\filldraw (4) circle (0.07);
\filldraw (5) circle (0.07);

\path[freccia] (1) edge node[fill=white] {$\infty$}  (2);
\path[freccia] (2) edge node[fill=white,sloped] {$\infty$}  (3);
\path[freccia] (2) edge node[fill=white,sloped] {$\infty$}  (4);
\path[freccia] (3) edge node[fill=white,sloped] {$\infty$}  (5);
\path[freccia] (4) edge node[fill=white,sloped] {$\infty$}  (5);

\end{tikzpicture}
& $X^6_q$ \\ \hline
\begin{tikzpicture}[baseline={([yshift=-3pt]current bounding box.center)},font=\footnotesize,scale=0.7]

\node[main node] (1) at (0,0) {};
\node[main node] (2) at (2,0) {};
\node[main node] (3) at (4,1) {};
\node[main node] (4) at (4,-1){};
\node at (0,-1.2) {};
\node at (0,1.2) {};

\filldraw (1) circle (0.07);
\filldraw (2) circle (0.07);
\filldraw (3) circle (0.07);
\filldraw (4) circle (0.07);

\path[freccia] (1) edge node[fill=white] {$\infty$}  (2);
\path[freccia] (2) edge node[fill=white,sloped] {$\infty$}  (3);
\path[freccia] (2) edge node[fill=white,sloped] {$\infty$}  (4);

\end{tikzpicture}
& $\C P^2_q\sqcup_{\C P^1_q}\C P^2_q$ \\ \hline
\begin{tikzpicture}[baseline={([yshift=-2pt]current bounding box.center)},font=\footnotesize,scale=0.7]

\node[main node] (1) at (0,0) {};
\node[main node] (2) at (2,0) {};
\node at (0,-0.3) {};
\node at (0,0.3) {};

\filldraw (1) circle (0.07);
\filldraw (2) circle (0.07);

\path[freccia] (1) edge node[fill=white] {$\infty$}  (2);

\end{tikzpicture}
& $\C P^1_q$ \\ \hline
\begin{tikzpicture}[baseline={([yshift=-2pt]current bounding box.center)},font=\footnotesize,scale=0.7]

\node[main node] (1) at (0,0) {};
\node at (0,-0.3) {};
\node at (0,0.3) {};

\filldraw (1) circle (0.07);

\end{tikzpicture}
& a point \\ \hline
\end{tabular}
\end{center}
\caption{Skeleta.}\label{tab:1}
\end{table}

The quantum space $X^6_q$ has C*-algebra $C(X^6_q)$ that, by definition, is the graph C*-algebra of the second amplified graph in Table \ref{tab:1}.
For $Gr_q(2,4)$, $\C P^1_q$ and the point, their description using amplified graph C*-algebras is well-known and/or described previously in this paper. Finally, the C*-algebra $C(\C P^2_q\sqcup_{\C P^1_q}\C P^2_q)$ is defined as a pullback
\begin{equation}\label{eq:comparalo}
\begin{tikzpicture}[scale=2.3]

\node (a1) at (135:1) {$C(\C P_q^2)$};
\node (a2) at ($(45:1)+(135:1)$) {$C(\C P_q^2\sqcup_{\C P_q^1}\C P_q^2)$};
\node (b1) at (0,0) {$C(\C P_q^1)$};
\node (b2) at (45:1) {$C(\C P_q^2)$};

\path[-To] (a2) edge (a1) (b2) edge (b1) (a1) edge (b1) (a2) edge (b2);

\begin{scope}[yshift=0.6cm]
\draw (135:0.2) -- (0,0) -- (45:0.2);
\end{scope}

\end{tikzpicture} ,
\end{equation}
where the *-homomorphism $C(\C P^2_q)\to C(\C P_q^1)$ is the standard one (see e.g.~\cite{ADHT22}).
We now prove that this pullback is isomorphic to the graph C*-algebra of the third amplified graph in Table \ref{tab:1}.

\begin{prop}\label{prop:51}
One has
\begin{equation}\label{eq:toproveX4}
C(\C P^2_q\sqcup_{\C P^1_q}\C P^2_q)\cong
C^*\!\!\left(
\begin{tikzpicture}[baseline={([yshift=-2pt]current bounding box.center)},font=\scriptsize,scale=0.7]

\node[main node] (1) at (0,0) {};
\node[main node] (2) at (2,0) {};
\node[main node] (3) at (4,1) {};
\node[main node] (4) at (4,-1){};

\filldraw (1) circle (0.07);
\filldraw (2) circle (0.07);
\filldraw (3) circle (0.07);
\filldraw (4) circle (0.07);

\path[freccia] (1) edge node[fill=white] {$\infty$}  (2);
\path[freccia] (2) edge node[fill=white,sloped] {$\infty$}  (3);
\path[freccia] (2) edge node[fill=white,sloped] {$\infty$}  (4);

\end{tikzpicture}\;
\right) \cong C^*(G_1)^{U(1)},
\end{equation}
where $G_1$ is the graph in Figure \ref{fig:X4}.
\end{prop}

\begin{proof}
It follows from \cite[Cor.~3.9]{ERS12} that we can add countably infinitely many arrows
from the first vertex on the left to the last two vertices in \eqref{eq:toproveX4} and get an isomorphic C*-algebra.
It follows then from Theorem \ref{thm:1} that the amplified graph C*-algebra in \eqref{eq:toproveX4} is isomorphic to $C^*(G_1)^{U(1)}$, where $G_1$ is the graph in Figure \ref{fig:X4}.

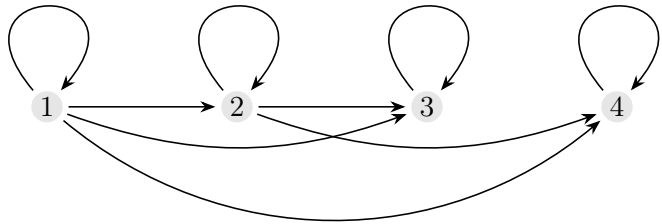
\begin{figure}[h]
\begin{tikzpicture}[inner sep=1pt,font=\small]

\clip (-0.7,1.5) rectangle (8.2,-1.7);
\node[main node] (1) {1};
\node (2) [main node,right of=1] {2};
\node (3) [main node,right of=2] {3};
\node (4) [main node,right of=3] {4};

\path[freccia] (1) edge[ciclo] (1);
\path[freccia] (2) edge[ciclo] (2);
\path[freccia] (3) edge[ciclo] (3);
\path[freccia] (4) edge[ciclo] (4);

\path[freccia] (1) edge (2);
\path[freccia] (1) edge[bend right=20] (3);
\path[freccia] (1) edge[bend right=40] (4);

\path[freccia] (2) edge (3);
\path[freccia] (2) edge[bend right=20] (4);

\end{tikzpicture}

\medskip

\caption{The graph $G_1$.}\label{fig:X4}
\end{figure}

Let $F_1$ be the subgraph of $G_1$ obtained by removing from it the vertex $3$ and all edges with target $3$,
and $F_2$ the one obtained by removing from it the vertex $4$ and all edges with target $4$.
Observe that both $F_1$ and $F_2$ are isomorphic to the graph of a quantum $5$-sphere (Figure \ref{fig:graphSnq}),
and $F_1\cap F_2$ is the graph of a quantum $3$-sphere.
The pair $\{F_1,F_2\}$ is an admissible decomposition of $G_4$ in the sense of \cite[Def.~2.1]{HRT18}.

\pagebreak

It follows from \cite[Theorem 3.1]{HRT18} that we have a $U(1)$-equivariant pullback diagram
\begin{center}
\begin{tikzpicture}[scale=2.5]

\node (a1) at (135:1) {$C(S^5_q)$};
\node (a2) at ($(45:1)+(135:1)$) {$C^*(G_1)$};
\node (b1) at (0,0) {$C(S^3_q)$};
\node (b2) at (45:1) {$C(S^5_q)$};

\path[-To] (a2) edge (a1) (b2) edge (b1) (a1) edge (b1) (a2) edge (b2);

\begin{scope}[yshift=0.6cm]
\draw (135:0.2) -- (0,0) -- (45:0.2);
\end{scope}

\end{tikzpicture} .
\end{center}
Passing to fixed-point subalgebras, we get the pullback diagram
\begin{center}
\begin{tikzpicture}[scale=2.5]

\node (a1) at (135:1) {$C(\C P^2_q)$};
\node (a2) at ($(45:1)+(135:1)$) {$C^*(G_1)^{U(1)}$};
\node (b1) at (0,0) {$C(\C P^1_q)$};
\node (b2) at (45:1) {$C(\C P^2_q)$};

\path[-To] (a2) edge (a1) (b2) edge (b1) (a1) edge (b1) (a2) edge (b2);

\begin{scope}[yshift=0.6cm]
\draw (135:0.15) -- (0,0) -- (45:0.15);
\end{scope}

\end{tikzpicture} .
\end{center}
By unicity of the pullback, comparing the latter diagram with \eqref{eq:comparalo} we deduce that 
$C^*(G_1)^{U(1)}$ is isomorphic to $C\smash{(\C P_q^2\sqcup_{\C P_q^1}\C P_q^2)}$.
\end{proof}

It follows from \cite[Cor.~3.9]{ERS12} and our Theorem \ref{thm:1} that $C(X_q^6)$ also has a presentation as the core of a Cuntz-Krieger algebra. More precisely, 

\begin{rem}\label{rem:G2}
$C(X_q^6)\cong C^*(G_2)^{U(1)}$, where $G_2$ is the graph in Figure \ref{fig:X6}.
\end{rem}

\begin{figure}[h]
\begin{tikzpicture}[inner sep=1pt,font=\small]

\clip (-0.7,1.5) rectangle (10.7,-2.8);
\node[main node] (1) {1};
\node (2) [main node,right of=1] {2};
\node (3) [main node,right of=2] {3};
\node (4) [main node,right of=3] {4};
\node (5) [main node,right of=4] {5};

\path[freccia] (1) edge[ciclo] (1);
\path[freccia] (2) edge[ciclo] (2);
\path[freccia] (3) edge[ciclo] (3);
\path[freccia] (4) edge[ciclo] (4);
\path[freccia] (5) edge[ciclo] (5);

\path[freccia] (1) edge (2);
\path[freccia] (1) edge[bend right=20] (3);
\path[freccia] (1) edge[bend right=40] (4);
\path[freccia] (1) edge[bend right=60] (5);

\path[freccia] (2) edge (3);
\path[freccia] (2) edge[bend right=20] (4);
\path[freccia] (2) edge[bend right=40] (5);

\path[freccia] (3) edge[bend right=20] (5);

\path[freccia] (4) edge (5);

\end{tikzpicture}

\medskip

\caption{The graph $G_2$.}\label{fig:X6}
\end{figure}
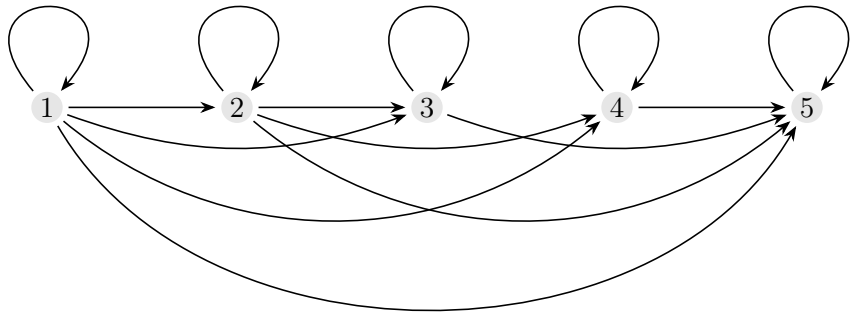

\subsection{The CW-structure}

For the CW-structure, it is useful to have a different presentation of both $C^*(L_{2,4})$ and $C^*(G_2)$ in terms of Toeplitz operators.

Let us introduce some notations first.
We denote by $(\ket{m})_{m\in\N}$ be the canonical basis of $\ell^2(\N)$, by $T$ the unilteral shift on $\ell^2(\N)$, given by
\[
T\ket{m}:=\ket{m+1} ,
\]
and we let $Q:=1-TT^*$ be the orthogonal projection onto $\ket{0}$ and $Q^\perp:=1-Q$. The unilateral shift generates the Toeplitz C*-algebra, that we denote here by $\mathcal{T}$. We denote by $\mathcal{K}\subset \mathcal{T}$ the ideal of compact operators, and by $\mathcal{T}^n$ and $\mathcal{K}^n$ the (spatial) tensor product of $n$ copies of $\mathcal{T}$ and $\mathcal{K}$, respectively. For decomposable tensors we shall use the standard leg numbering notation. Thus, for example $T_i\in\mathcal{T}^n$ is the operator
\[
\underbrace{1\otimes \ldots\otimes 1}_{i-1\text{ times}}\otimes T
\otimes \underbrace{1\otimes \ldots\otimes 1}_{n-i\text{ times}} ,
\]
etc. Note that
\begin{equation}\label{eq:relle}
T^*T=1 , \qquad
QT=T^*Q=0 .
\end{equation}

\begin{df}
Let $A=(a_{ij})$ be the adjacency matrix of the graph $L_{2,4}$. We denote by $\mathscr{A}(M^9_0)$ the free *-algebra
generated by elements $\{Z_i,Z_i^*:1\leq i\leq 6\}$ with relations
\begin{subequations}\label{eq:Zrelations}
\begin{alignat}{2}
Z_iZ_j &=0 && \forall\;i,j\;\text{such that}\;a_{ij}=0 , \label{eq:noref} \\
Z_i^*Z_j &=0 && \forall\;i\neq j , \label{eq:orthogonal} \\
Z_i^*Z_i &=\sum\nolimits_{j=1}^6a_{ij}Z_jZ_j^* , \label{eq:projections} \\
\sum\nolimits_{j=1}^6Z_jZ_j^* &=1 . \label{eq:sphere}
\end{alignat}
\end{subequations}
\end{df}
\noeqref{eq:noref}

Because of \eqref{eq:sphere}, the universal C*-seminorm is well-defined and the C*-enveloping algebras of $\mathscr{A}(M^9_0)$, which we denote by $C(M^9_0)$, exists (generators have norm $\leq 1$, and polynomials have finite norm, in any *-representation on a Hilbert space).

\begin{df}
We denote by $C(M^7_0)$ the quotient of $C(M^9_0)$ by the ideal generated by $Z_6$. The class of $Z_i$ in the quotient algebra will be denoted by $Y_i$, $1\leq i\leq 5$.
\end{df}

Note that by ideal in a C*-algebra we always mean a closed two-sided $*$-ideal.

We choose the notation $M^n_0$ to suggest that this is the $q\to 0$ limit of a family of $q$-deformation of some $n$-dimensional manifold $M^n$. The study of these $q$-deformations will be the topic of a future work.
Both $C(M^9_0)$ and $C(M^7_0)$ are $\Z$-graded, with grading assigning degree $+1$ to each $Z_i$ and $Y_i$ and $-1$ to $Z_i^*$ and $Y_i^*$, and there is an obvious $U(1)$-action inducing this grading.
The quotient map $C(M^9_0)\to C(M^7_0)$ is, by construction, $U(1)$-equivariant.

Observe that $L_{2,4}$ is very close to the graph of the $11$-dimensional quantum sphere $S^{11}_q$ (see Figure \ref{fig:graphSnq}): we get the graph of $S^{11}_q$ by adding to $L_{2,4}$ one extra arrow from the vertex $3$ to the vertex $4$.
The relations of $\mathscr{A}(M^9_0)$ are also very close to those of $\mathscr{A}(S^{11}_q)$ at $q=0$ (see e.g.~\cite[Eq.~(3.2)]{Dan23}).
Compared to $S^{11}_q$ at $q=0$, here we have an extra relation
\[
Z_3Z_4=0 ,
\]
which one could think as a kind of $q\to 0$ limit of a $q$-Pl{\"u}cker equation. 
We can think of the algebra $\mathscr{A}(M^9_0)$ as a kind of quantization of the submanifold of $S^{11}$ defined by the equation \eqref{eq:pluck}.

\begin{lemma}\label{lemma:54}
A $*$-homomorphism $\pi:C(M^9_0)\to\mathcal{T}^{4}\otimes C(S^1)$ is given on generators by
\begin{align*}
\pi(Z_1) &:= T_1 \otimes z,
\\
\pi(Z_2) &:= Q_1T_2T_3 \otimes z,
\\
\pi(Z_3) &:= Q_1T_2Q_3 \otimes z,
\\
\pi(Z_4) &:= Q_1Q_2T_3 \otimes z,
\\
\pi(Z_5) &:= Q_1Q_2Q_3T_4 \otimes z,
\\
\pi(Z_6) &:= Q_1Q_2Q_3Q_4 \otimes z,
\end{align*}
where $T$ and $Q$ are the Toeplitz and compact operators defined above, and $z$ is the unitary generator of $C(S^1)$.
\end{lemma}

\begin{proof}
Using \eqref{eq:relle} one easily checks that the operators $Z'_i:=\pi(Z_i)$ and
$Z_i'^*=\pi(Z_i)^*$ satisfy the relations \eqref{eq:Zrelations}.
\end{proof}

\begin{prop}\label{prop:M09}
\begin{enumerate}
\item
There is a $U(1)$-equivariant isomorphism $\psi:C^*(L_{2,4})\to C(M^9_0)$ given on generators by:
\begin{equation}\label{eq:psi}
\psi(P_i)=Z_iZ_i^* , \qquad\quad \psi(S_{i,j})=Z_iZ_jZ_j^* .
\end{equation}
\item
The map $\pi$ in Lemma \ref{lemma:54} is injective.
\end{enumerate}
\end{prop}

\begin{proof}
If we multiply \eqref{eq:sphere} on both sides by $Z_i$ from the right and use \eqref{eq:orthogonal} we find:
\[
Z_i\stackrel{\eqref{eq:sphere}}{=}\sum\nolimits_{j=1}^6Z_jZ_j^*Z_i\stackrel{\eqref{eq:orthogonal}}{=}Z_iZ_i^*Z_i ,
\]
proving that $Z_i$ is a partial isometry, and $Z_iZ_i^*$ is a projection. From \eqref{eq:orthogonal} it follows that 
these projections are orthogonal: $(Z_iZ_i^*)(Z_jZ_j^*)=0$ for $i\neq j$.

We now check that $\{\psi(P_i),\psi(S_{i,j})\}$ in \eqref{eq:psi} is a Cuntz--Krieger $L_{2,4}$-family, so that the *-homomorphism $\psi$ is well-defined. We already know that $\{\psi(P_i)\}$ is a family of orthogonal projections. Moreover, for every arrow $(i,j)$ in $L_{2,4}$,
\[
\psi(S_{i,j})^*\,\psi(S_{i,j})=Z_jZ_j^*(Z_i^*Z_i)Z_jZ_j^*\stackrel{\eqref{eq:projections}}{=}
\sum\nolimits_{k=1}^6a_{ik}Z_jZ_j^*Z_kZ_k^* Z_jZ_j^*\stackrel{\eqref{eq:orthogonal}}{=}
a_{ij}Z_jZ_j^*Z_jZ_j^* Z_jZ_j^* .
\]
Now $a_{ij}=1$ since $(i,j)$ is an arrow, and we get:
\[
\psi(S_{i,j})^*\,\psi(S_{i,j})=\psi(P_j)^3=\psi(P_j) .
\]
As a consequence,
\[
\psi(S_{i,j})\psi(S_{i,j})^*\,\psi(S_{i,j})=\psi(S_{i,j})\psi(P_j)=Z_i(Z_jZ_j^*)^2=Z_iZ_jZ_j^*=\psi(S_{i,j}) ,
\]
so that $\psi(S_{i,j})$ is a partial isometry. Finally,
\begin{multline*}
\sum\nolimits_{j=1}^6a_{ij}\psi(S_{i,j})\psi(S_{i,j})^*=
\sum\nolimits_{j=1}^6a_{ij}Z_i(Z_jZ_j^*)^2Z_i^* \\
=Z_i\Big(\sum\nolimits_{j=1}^6a_{ij}Z_jZ_j^*\Big)Z_i^*
\stackrel{\eqref{eq:projections}}{=}Z_iZ_i^*Z_iZ_i^*=Z_iZ_i^*=\psi(P_i) .
\end{multline*}
But the left hand side is exactly the sum over all arrows starting at $i$. Hence, all the Cuntz-Krieger relations are satisfied.

From \eqref{eq:projections},
\[
\sum\nolimits_{j=1}^6a_{ij}\psi(S_{i,j})=Z_i\sum\nolimits_{j=1}^6a_{ij}Z_jZ_j^*=Z_iZ_i^*Z_i=Z_i .
\]
Thus, the image of $\psi$ contains all the generators of $C(M^9_0)$, proving that $\psi$ is surjective.

It remains to prove that $\psi$ is injective.
Notice that the $*$-homomorphism in Lemma \ref{lemma:54} is $U(1)$-equivariant w.r.t.~the $U(1)$-action on $\mathcal{T}^{4}\otimes C(S^1)$ coming from the obvious action on the rightmost factor. The composition $\pi\circ\psi:C^*(L_{2,4})\to\mathcal{T}^{4}\otimes C(S^1)$
sends the vertex projections to the following elements
\begin{align*}
\pi(\psi(P_1)) &= Q_1^\perp \otimes 1,
\\
\pi(\psi(P_2)) &= Q_1Q_2^\perp Q_3^\perp \otimes 1,
\\
\pi(\psi(P_3)) &= Q_1Q_2^\perp Q_3 \otimes 1,
\\
\pi(\psi(P_4)) &= Q_1Q_2Q_3^\perp  \otimes 1,
\\
\pi(\psi(P_5)) &= Q_1Q_2Q_3Q_4^\perp \otimes 1,
\\
\pi(\psi(P_6)) &= Q_1Q_2Q_3Q_4 \otimes 1 .
\end{align*}
Since these are all different from $0$, from the Gauge-Invariant Uniqueness Theorem (see e.g.~\cite{R05}) we deduce that $\pi\circ\psi$ is injective, and then $\psi$ is injective. This proves that $\psi$ is an isomorphism, hence $\pi=(\pi\circ\psi)\circ \psi^{-1}$ is injective as well.
\end{proof}

\begin{lemma}\label{lemma:57}
\begin{enumerate}
\item\label{en:571}
There is a $U(1)$-equivariant diagram of C*-algebras
\begin{equation}\label{eq:pb57}
\begin{tikzpicture}[>=To,xscale=4,yscale=2]

\node (f) at (0,2) {$C^*(L_{2,4})$};
\node (g) at (1,2) {$C^*(G_2)$};
\node (a) at (0,1) {$C(M^9_0)$};
\node (b) at (0,0) {$\mathcal{T}^{4}\otimes C(S^1)$};
\node (c) at (1,1) {$C(M^7_0)$};
\node (d) at (1,0) {$\mathcal{T}^{4}/\mathcal{K}^{4}\otimes C(S^1)$};

\draw[->>] (f) -- (g);
\draw[->>] (a) --node [above,font=\footnotesize] {$q$} (c);
\draw[->>] (b) -- (d);
\draw[right hook->] (a) -- node[left,font=\footnotesize] {$\smash[t]{\pi}$} (b);
\draw[right hook->] (c) -- node[right,font=\footnotesize] {$\smash[t]{\widetilde{\pi}}$} (d);
\draw[->] (f) -- node[left,font=\footnotesize] {$\smash[t]{\psi}$} (a);
\draw[->] (g) -- node[right,font=\footnotesize] {$\smash[t]{\widetilde{\psi}}$} (c);

\end{tikzpicture} ,
\end{equation}
where $\pi$ is the map in Lemma \ref{lemma:54}, $\psi$ is the map in Prop.~\ref{prop:M09}, the $U(1)$ action on the C*-algebras in the third row is on the rightmost factor and the horizontal arrows are canonical quotient maps.

\item\label{en:572} The map $\widetilde{\psi}$ is an isomorphism.

\item\label{en:573} The map $\pi$ restricts to a bijection between $\ker(q)$ and $\mathcal{K}^4\otimes C(S^1)$.

\item\label{en:574} The map $\widetilde\pi$ is injective.

\item\label{en:575} The outer rectangle and the inner squares in \eqref{eq:pb57} are pullback diagrams.
\end{enumerate}
\end{lemma}

\begin{proof}
We already established in Lemma \ref{lemma:54} that $\pi$ is injective and $\psi$ is an isomorphism.

The C*-algebra $C^*(G_2)$ is the quotient of $C^*(L_{2,4})$ by the ideal generated by the vertex projection $P_6$. The C*-algebra $C(M^7_0)$ is the quotient of $C(M^9_0)$ by the ideal generated by $Z_6$, which is the same as ideal generated by $Z_6Z_6^*$,
since $Z_6=(Z_6Z_6^*)Z_6$. Since $\psi$ maps $P_6$ to $Z_6Z_6^*$, it maps one ideal to the other and induces a well-defined isomorphism $\widetilde{\psi}$ between quotient algebras. This proves \ref{en:572}.

The *-homomorphism $\pi$ maps $Z_6$ into $\mathcal{K}^4\otimes C(S^1)$.
So, it maps the kernel of the quotient map $q:C(M^9_q)\to C(M^7_q)$ into 
$\mathcal{K}^4\otimes C(S^1)$ and induces a well-defined *-homomorphism $\widetilde{\pi}$ between quotient algebras. This completes the proof of \ref{en:571}.

We now pass to point \ref{en:573}. 
We already know that $\pi(\ker q)\subseteq\mathcal{K}^4\otimes C(S^1)$.
We need to show the opposite inclusion.
Since $\pi$ is injective, in the rest of the proof we omit the representation symbol $\pi$ and identify $Z_i$ with $\pi(Z_i)$.
For $n_1,\ldots,n_4\in\N$, define the following partial isometries in $\mathcal{T}^4\otimes C(\mathbb{S}^1)$:
\[
V(n_1,\ldots,n_4) =\begin{cases}
Z_1^{n_1}Z_2^{n_2}Z_4^{n_3-n_2}Z_5^{n_4} & \text{if }n_2\leq n_3 ,
\\[2pt]
Z_1^{n_1}Z_2^{n_3}Z_3^{n_2-n_3}Z_5^{n_4} & \text{if }n_2\geq n_3 .
\end{cases}
\]
Let
\[
R_{n_1,\ldots,n_4}^{m_1,\ldots,m_4}:=T_1^{n_1}T_2^{n_2}T_3^{n_3}T_4^{n_4}
Q_1Q_2Q_3Q_4
(T_1^{m_1}T_2^{m_2}T_3^{m_3}T_4^{m_4})^* .
\]
If $n_1,\ldots,n_4,m_1,\ldots,m_4$ are not all zero, then
\[
R_{n_1,\ldots,n_4}^{m_1,\ldots,m_4}\otimes 1=V(n_1,\ldots,n_4)Z_6^{m_1+\max\{m_2,m_3\}+m_4}(Z_6^*)^{n_1+\max\{n_2,n_3\}+n_4}V(m_1,\ldots,m_4) ,
\]
while
\[
R_{0,\ldots,0}^{0,\ldots,0}\otimes 1=Z_6Z_6^* .
\]
In both cases, as shown by the above formulas,
$R_{n_1,\ldots,n_4}^{m_1,\ldots,m_4}\otimes 1$ belongs to the ideal of $C(M^9_0)$ generated by $Z_6$, that is $\ker(q)$. 

It is not difficult to see that, as an operator on $\ell^2(\N^4)$, one has
\begin{equation}\label{eq:rank1}
R_{n_1,\ldots,n_4}^{m_1,\ldots,m_4}=\ketbra{n_1,\ldots,n_4}{m_1,\ldots,m_4} ,
\end{equation}
where $\ket{n_1,\ldots,n_4}=\ket{n_1}\otimes\ldots\otimes\ket{n_4}$ is the canonical basis of $\ell^2(\N)\otimes\ell^2(\N)\otimes\ell^2(\N)\otimes\ell^2(\N)$, 
that we identify with $\ell^2(\N^4)$,
and we use the ``bra-ket'' notation $\ketbra{v_1}{v_2}$ for the rank $1$ operator $w\mapsto \inner{v_2,w}v_1$.

Since the rank $1$ operators in \eqref{eq:rank1} generate $\mathcal{K}^{\otimes 4}=\mathcal{K}(\ell^2(\N^4))$ (the C*-algebra of compact operators on $\ell^2(\N^4)$ up to the above-mentioned identification), we deduce that any operator of the form $b\otimes 1$, with $b$ compact belongs, to $\pi(\ker q)$.
Any finite rank operator can be written as a finite sum $\sum_ia_iR_{0,\ldots,0}^{0,\ldots,0}b_i$
with $a_i,b_i$ compact ($a_i=\ketbra{v_i}{0,0,0,0}$ and $b_i=\ketbra{0,0,0,0}{w_i}$ for some unit vectors $v_i,w_i$). From
\[
a_iR_{0,\ldots,0}^{0,\ldots,0}b_i\otimes z^k=\begin{cases}
(a_i\otimes 1)
Z_6^k(b_i\otimes 1) & \text{if }k\geq 0 ,\\
(a_i\otimes 1)
(Z_6^*)^{-k}(b_i\otimes 1) & \text{if }k<0 ,
\end{cases}
\]
we deduce that $\pi(\ker q)$ contains the tensor product of finite rank operators on $\ell^2(\N^4)$ with the algebra of polynomial functions in $z,z^*$. Since this tensor product is dense in $\mathcal{K}^{\otimes 4}\otimes C(\mathbb{S}^1)$, this concludes the proof of \ref{en:573}.

Concerning point \ref{en:574}, from the diagram \eqref{eq:pb57} we see that:
\[
\ker(\widetilde\pi)=q\Big(\pi^{-1}\big(\mathcal{K}^{\otimes 4}\otimes C(\mathbb{S}^1)\big)\Big) .
\]
But from point \ref{en:573} we know that $\pi^{-1}\big(\mathcal{K}^{\otimes 4}\otimes C(\mathbb{S}^1)\big)=\ker q$, thus $\ker(\widetilde\pi)=0$.

It remains to prove point \ref{en:575}. It is enough to show that the bottom square is a pullback, since $\psi$ and $\widetilde\psi$ are isomorphism. For a commutative square of C*-algebras,
\begin{center}
\begin{tikzpicture}[>=To,scale=2]

\node (a) at (0,1) {$X$};
\node (b) at (1,1) {$B$};
\node (c) at (0,0) {$A$};
\node (d) at (1,0) {$C$};

\draw[->] (a) -- node[left] {$\delta$} (c);
\draw[->] (a) -- node[above] {$\gamma$} (b);
\draw[->] (b) -- node[right] {$\beta$} (d);
\draw[->] (c) -- node[below] {$\alpha$} (d);

\end{tikzpicture} \raisebox{-5pt}{,}
\end{center}
necessary and sufficient conditions to have a pullback diagram are given in \cite[Prop.~3.1]{P99}. If $\alpha$ and $\gamma$ are surjective and $\delta$ is injective, simple set-theoretic arguments show that the conditions in \cite[Prop.~3.1]{P99} reduce to the single condition that
\begin{equation}\label{eq:immediatelyfollows}
\ker\alpha\subseteq\delta(\ker\gamma) .
\end{equation}
In the present case, $\gamma=q$, $\delta=\pi$, $\ker\alpha=\mathcal{K}^{\otimes 4}\otimes C(\mathbb{S}^1)$, and \eqref{eq:immediatelyfollows} immediately follows from \ref{en:573}.
\end{proof}

\begin{lemma}\label{lemma:58}
There is a $U(1)$-equivariant pullback diagram of C*-algebras
\begin{equation}\label{eq:pb58}
\begin{tikzpicture}[>=To,xscale=3.5,yscale=2]

\node (a) at (0,1) {$C^*(G_2)$};
\node (b) at (1,1) {$C^*(G_1)$};
\node (c) at (0,0) {$\mathcal{T}^{3}\otimes C(S^1)$};
\node (d) at (1,0) {$\mathcal{T}^{3}/\mathcal{K}^{3}\otimes C(S^1)$};

\draw[->>] (a) -- (b);
\draw[->>] (c) -- (d);
\draw[right hook->] (a) -- (c);
\draw[right hook->] (b) -- (d);

\end{tikzpicture} .
\end{equation}
\end{lemma}

\begin{proof}
The proof is completely analogous to that of Lemma \ref{lemma:57}, and we will only sketch it. For starters, $C^*(G_2)\cong C(M^7_0)$ is generated by the class $Y_i$ of $Z_i$, for $1\leq i\leq 5$. They satisty the relations obtained from \eqref{eq:Zrelations} by putting $Z_6$ equal to $0$.
In parallel with Lemma \ref{lemma:54}, one construct a map to $\mathcal{T}^3\otimes C(S^1)$ given on generators by
\begin{align*}
Y_1 &\mapsto T_1 \otimes z,
\\
Y_2 &\mapsto Q_1T_2T_3 \otimes z,
\\
Y_3 &\mapsto Q_1T_2Q_3 \otimes z,
\\
Y_4 &\mapsto Q_1Q_2T_3 \otimes z,
\\
Y_5 &\mapsto Q_1Q_2Q_3 \otimes z.
\end{align*}
One checks that the defining relations of $C(M^7_0)$ are satisfied.

This leads to the diagram \eqref{eq:pb58}, where the horizontal maps are the canonical quotient maps ($C^*(G_1)$ is a quotient of $C^*(G_2)$ by a gauge-invariant ideal). Next, one proves that this is a pullback diagram by repeating almost verbatim the proof of Lemma \ref{lemma:57}.
\end{proof}

We can now prove the main theorem of this section.

\begin{thm}\label{thm:qCW}
There is a sequence of quotient maps
\begin{equation}\label{eq:qCW}
\C\longleftarrow C(\C P_q^1)\longleftarrow C(\C P_q^2\sqcup_{\C P_q^1}\C P_q^2) \longleftarrow 
C(X_q^6)\longleftarrow C(Gr_q(2,4)) ,
\end{equation}
where the skeleta are (isomorphic to) the amplified graph C*-algebras in Table \ref{tab:1}.
Each morphism in \eqref{eq:qCW} is the top arrow in one of the following commutative diagrams:
\begin{equation}\label{eq:attachingcells}
\begin{array}{rc}
\begin{tikzpicture}[xscale=3.5,yscale=2,baseline=(current bounding box.north)]

\node (a1) at (0,1) {$\C$};
\node (a2) at (1,1) {$C\smash{(\C P_q^1)}$};
\node (b1) at (0,0) {$C\smash{(S^1)}$};
\node (b2) at (1,0) {$C\smash{(B^2_q)}$};

\path[-To] (a2) edge (a1) (b2) edge (b1) (a1) edge (b1) (a2) edge (b2);

\node at (0.5,0.5) {\textup{(I)}};

\end{tikzpicture}
&
\begin{tikzpicture}[scale=2.3,inner sep=1pt,baseline=(current bounding box.north)]

\node (a1) at (155:1) {$C(\C P_q^1)$};
\node (a2) at (0,0) {$C(\C P_q^2)$};
\node (a3) at (25:1) {$C(\C P_q^2\sqcup_{\smash{\C P_q^1}}\C P_q^2)$};
\node (b2) at (245:1) {$C(B_q^4)$};
\node (b3) at (-65:1) {$C(S^3_q)$};
\node (b1) at ($(a1)+(b2)$) {$C(S^3_q)$};
\node (b4) at ($(a3)+(b3)$) {$C(B_q^4)$};

\path[-To]
	(a2) edge (a1) (a3) edge (a2) (a3) edge (a1)
	(b4) edge (b3) (b2) edge (b1);
\path[-To,shorten >=3pt]
	(a1) edge (b1) (a2) edge (b2) (a2) edge (b3) (a3) edge (b4);

\node at (270:0.55) {\textup{(II)}};

\end{tikzpicture}
\\
\begin{tikzpicture}[xscale=3.5,yscale=2,baseline=(current bounding box.north)]

\node (a1) at (0,1) {$C\smash{(\C P_q^2\sqcup_{\C P_q^1}\C P_q^2)}$};
\node (a2) at (1,1) {$C(\smash{X^6_q})$};
\node (b1) at (0,0) {$\mathcal{T}^3/\mathcal{K}^3$};
\node (b2) at (1,0) {$\mathcal{T}^3$};

\path[-To] (a2) edge (a1) (b2) edge (b1) (a1) edge (b1) (a2) edge (b2);

\node at (0.5,0.5) {\textup{(III)}};

\end{tikzpicture}
&
\begin{tikzpicture}[xscale=3.5,yscale=2,baseline=(current bounding box.north)]

\node (a1) at (0,1) {$C\smash{(X_q^6)}$};
\node (a2) at (1,1) {$C\smash{(Gr_q(2,4))}$};
\node (b1) at (0,0) {$\mathcal{T}^4/\mathcal{K}^4$};
\node (b2) at (1,0) {$\mathcal{T}^4$};

\path[-To] (a2) edge (a1) (b2) edge (b1) (a1) edge (b1) (a2) edge (b2);

\node at (0.5,0.5) {\textup{(IV)}};

\end{tikzpicture}
\end{array}
\end{equation}
The rectangles (I), (III) and (IV) are pullback diagrams, and the two squares in (II) are pullbacks as well.
\end{thm}

\begin{proof}
The pullback diagram (I) in \eqref{eq:attachingcells} is a special case of the pullback diagram
\begin{center}
\begin{tikzpicture}[xscale=3,yscale=1.8]

\node (a1) at (0,1) {$C\smash{(\C P_q^{n-1})}$};
\node (a2) at (1,1) {$C\smash{(\C P_q^n)}$};
\node (b1) at (0,0) {$C\smash{(S^{2n-1}_q)}$};
\node (b2) at (1,0) {$C\smash{(B^{2n}_q)}$};

\path[-To] (a2) edge (a1) (b2) edge (b1) (a1) edge (b1) (a2) edge (b2);

\draw (0.5,0.6) -- (0.5,0.45) -- (0.6,0.45);

\end{tikzpicture}
\end{center}
in \cite[Prop.~4.1]{ADHT22}, valid for every $n\geq 1$.
For $n=2$ this is exactly the left square in (II).
Gluing this diagram with \eqref{eq:comparalo} we get
\begin{center}
\hspace*{1cm}\begin{tikzpicture}[xscale=3,yscale=1.8]

\node (a1) at (0,1) {$C\smash{(\C P_q^2)}$};
\node (a2) at (1,1) {$C\smash{(\C P_q^2\sqcup_{\C P_q^1}\C P_q^2)}$};
\node (b1) at (0,0) {$C\smash{(\C P_q^1)}$};
\node (b2) at (1,0) {$C\smash{(\C P_q^2)}$};
\node (c1) at (0,-1) {$C\smash{(S^3_q)}$};
\node (c2) at (1,-1) {$C\smash{(B^4_q)}$};

\path[-To] (a2) edge (a1) (b2) edge (b1) (a1) edge (b1) (a2) edge (b2) (c2) edge (c1) (b1) edge (c1) (b2) edge (c2);

\draw (0.5,0.6) -- (0.5,0.45) -- (0.6,0.45);
\draw (0.5,-0.4) -- (0.5,-0.55) -- (0.6,-0.55);

\end{tikzpicture} .
\end{center}
The outer rectangle is a pullback, and is exactly the right square in (II).

From the pullbacks in Lemma \ref{lemma:58} and \ref{lemma:57}, respectively, passing to fixed point subalgebras for the action of $U(1)$, we get the pullback diagrams:
\begin{center}
\begin{tikzpicture}[xscale=3.5,yscale=2,baseline=(current bounding box.north)]

\node (a1) at (0,1) {$C^*(G_1)^{U(1)}$};
\node (a2) at (1,1) {$C^*(G_2)^{U(1)}$};
\node (b1) at (0,0) {$\mathcal{T}^3/\mathcal{K}^3$};
\node (b2) at (1,0) {$\mathcal{T}^3$};

\path[-To] (a2) edge (a1) (b2) edge (b1) (a1) edge (b1) (a2) edge (b2);

\draw (0.5,0.6) -- (0.5,0.45) -- (0.6,0.45);

\end{tikzpicture}
\qquad
\begin{tikzpicture}[xscale=3.5,yscale=2,baseline=(current bounding box.north)]

\node (a1) at (0,1) {$C^*(G_2)^{U(1)}$};
\node (a2) at (1,1) {$C^*(L_{2,4})^{U(1)}$};
\node (b1) at (0,0) {$\mathcal{T}^4/\mathcal{K}^4$};
\node (b2) at (1,0) {$\mathcal{T}^4$};

\path[-To] (a2) edge (a1) (b2) edge (b1) (a1) edge (b1) (a2) edge (b2);

\draw (0.5,0.6) -- (0.5,0.45) -- (0.6,0.45);

\end{tikzpicture}
\end{center}
But 
$C^*(L_{2,4})^{U(1)}\cong C(Gr_q(2,4))$,
$C^*(G_1)^{U(1)}\cong C(\C P^2_q\sqcup_{\C P^1_q}\C P^2_q)$
and $C^*(G_2)^{U(1)}\cong C(X_q^6)$
(Cor.~\ref{cor:geom}, Prop.~\ref{prop:51}, Rem.~\ref{rem:G2}). Thus, we get the pullbacks (III) and (IV).
\end{proof}

The arrows in \eqref{eq:qCW} are written right-to-left to make the correspondence with the classical CW-structure \eqref{eq:classicalCW} more clear. We interpret each pullback diagram in \eqref{eq:attachingcells} as ``attaching'' a cell (the quantum space on the bottom-right corner) to a skeleton in \eqref{eq:qCW} to obtain the next one. In particular, in (II) we first attach a $4$-cell to the $2$-skeleton $\C P^1_q$ to get $\C P^2_q$, and then attach another $4$-cell to $\C P^2_q$ to get the $4$-skeleton $\C P_q^2\sqcup_{\C P_q^1}\C P_q^2$.
The passage from the $2$-skeleton to the $4$-skeleton can be equivalently described as the following multipullback
\begin{equation}\label{eq:multi}
\begin{tikzpicture}[scale=2]

\node (a1) at (2,1) {$C(\C P_q^2\sqcup_{\C P_q^1}\C P_q^2)$};

\node (b1) at (0,0) {$C(B^4_q)$};
\node (b2) at (2,0) {$C(\C P_q^1)$};
\node (b3) at (4,0) {$C(B^4_q)$};

\node (c1) at (1,-0.7) {$C(S^3_q)$};
\node (c2) at (3,-0.7) {$C(S^3_q)$};

\node (d) at (2,-1.5) {$C(S^3_q)$};

\path[-To] (a1) edge (b1) (a1) edge (b2) (a1) edge (b3) (b1) edge (c1) (b2) edge (c1) (b2) edge (c2) (b3) edge (c2) (b1) edge[bend right] (d) (b3) edge[bend left] (d);

\end{tikzpicture},
\end{equation}
where we attach (in one step) two $4$-cells to $\C P^1_q$. Indeed, we can attach to the diagram (II) in \eqref{eq:attachingcells}
a trivial pullback diagram for $S^3_q$ and get the commutative diagram
\begin{center}
\begin{tikzpicture}[scale=2.5]

\node (a) at (30:1) {$C(\C P_q^2\sqcup_{\smash{\C P_q^1}}\C P_q^2)$};

\node (b1) at (0,0) {$C(\C P_q^2)$};
\node (b2) at (225:1) {$C(B_q^4)$};
\node (b3) at (-45:1) {$C(\C P_q^1)$};
\node (b4) at ($(b2)+(b3)$) {$C(S^3_q)$};

\node (c1) at (3,0) {$C(B_q^4)$};
\node (c2) at ($(c1)-(a)$) {$C(S^3_q)$};

\node (d1) at ($(c2)+(225:1)$) {$C(S^3_q)$};
\node (d3) at ($(c2)+(315:1)$) {$C(S^3_q)$};
\node (d2) at ($(d1)+(d3)-(c2)$) {$C(S^3_q)$};

\path[-To]
	(a) edge (b1) (a) edge (c1) (b1) edge (c2) (c1) edge (c2)
	(b1) edge (b2) (b1) edge (b3) (b2) edge (b4) (b3) edge (b4)
	(c2) edge (d1) (c2) edge (d3) (d1) edge node[above,sloped] {$=$} (d2) (d3) edge node[above,sloped] {$=$} (d2);

\end{tikzpicture},
\end{center}
where each rhombus is a pullback. From \cite[Lemma 0.2]{Rud12} we get the multipullback diagram \eqref{eq:multi}.

Finally, we pass from the 4-skeleton to the 6-skeleton and from the 6-skeleton to the 8-skeleton by attaching ``Toeplitz balls'' of the appropriate dimension along their boundary. Here, we interpret $\mathcal{T}$ as ``functions'' on a quantum disk, and $\mathcal{T}^n$ as the Cartesian product of $n$ quantum disks, that in the classical case is a topological space homeomorphic to a $2n$-ball. The ``boundary'' $\mathcal{T}^n/\mathcal{K}^n$ is isomorphic to the C*-algebra of the multipullback quantum sphere $S^{2n-1}_H$, and has the same K-theory as the $2n-1$ dimensional sphere  \cite{hnpsz18}.
We conclude that,

\begin{cor}
The sequence \eqref{eq:qCW} is a strict CW-filtration in the sense of \cite{DHMSZ20}.
\end{cor}

\begin{rem}\label{rem:Ktheory}
For the quotient maps $C(\C P_q^2\sqcup_{\C P_q^1}\C P_q^2) \longleftarrow 
C(X_q^6)\longleftarrow C(Gr_q(2,4))$ one can also get two pullback diagrams, different from the ones in \eqref{eq:attachingcells}, by using the result \cite[Cor.~2.5]{ADHT22} about trimmable graph C*-algebras. However, in doing so one gets pullback diagrams where the C*-algebras in the bottom don't have the K-theory of an odd-dimensional sphere and an even-dimensional closed ball. Thus, they cannot be interpreted as ``attaching cells'' to noncommutative spaces.
\end{rem}

\appendix
\renewcommand{\sectionname}{}
\section{On the dimension group of a complex projective space}\label{app:CPn}

Let $n\geq 2$. A classical result by Atiyah and Todd \cite{AT60} states that
\[
K^0(\C P^{n-1})=\Z[x]/(x^n) ,
\]
where $x:=1-[\mathcal{L}_1]$ is the Euler class of the Hopf line bundle $\mathcal{L}_1$. In general, for $k\in\Z$, we denote by $\mathcal{L}_k\to\C P^{n-1}$ the line bundle associated to the $U(1)$-principal bundle $S^{2n-1}\to\C P^{n-1}$ and to the representation $z\mapsto z^k$ of $U(1)$.
From $[\mathcal{L}_k]=[\mathcal{L}_1]^k$, $[\mathcal{L}_1]=1-x$, and the binomial formula it follows that
\begin{equation}\label{eq:binomial}
[\mathcal{L}_k]=\sum_{i=0}^{n-1}\binom{k}{i}(-x)^i \qquad\forall\; k\in\N ,
\end{equation}
with the convention that $\binom{k}{i}=0$ if $i>k$. This formula is valid for $k<0$ as well, if we interpret the binomial as
\[
\binom{k}{i}:=\frac{k(k-1)\cdots (k-i+1)}{i!}=(-1)^i\binom{|k|+i-1}{i}.
\]

\begin{prop}~
\begin{enumerate}
\item\label{en:1}
For every $n\geq 2$,
\begin{equation}\label{eq:inclusion}
K^0(\C P^{n-1})^+\subseteq\{0\}\cup\{a_0+a_1x+\ldots+a_{n-1}x^{n-1}\,|\,a_i\in\Z,a_0\geq 1 \} .
\end{equation}

\item\label{en:2}
The inclusion \eqref{eq:inclusion} is an equality if $n=2$.

\item\label{en:3}
The inclusion \eqref{eq:inclusion} is proper if $n\geq 3$.
In particular, $1+x\notin K^0(\C P^{n-1})^+$ if $n\geq 3$.
\end{enumerate}
\end{prop}

\begin{proof}
Point \ref{en:1} follows from the fact that for a vector bundle $V\to \C P^{n-1}$,
$[V]=a_0+O(x)$ where $a_0$ is the rank of $V$. Hence, $a_0\geq 0$ for every element in the positive cone, and $a_0=0$ implies that $[V]=0$ (the only rank $0$ vector bundle is the trivial one).

To prove point \ref{en:2}, observe that $0\in K^0(\C P^1)^+$. Moreover, for every $a_0,a_1\in\Z$ with $a_0\geq 1$, one has
\[
a_0+a_1x=(a_0-1)[\mathcal{L}_0]+[\mathcal{L}_{-a_1}] \;,
\]
which proves that $a_0+a_1x\in K^0(\C P^1)^+$.

Finally, let $n\geq 3$. To prove point \ref{en:3} it is enough to show that $1+x\notin K^0(\C P^{n-1})^+$. From \eqref{eq:binomial} we deduce that
\[
[\mathcal{L}_k]=1-kx+\frac{k(k-1)}{2}x^2+O(x^3)
\]
has coefficient of $x^2$ which is non-negative, for all $k\in\Z$.

Any vector bundle $V\to \C P^{n-1}$ is stably isomorphic to a direct sum of line bundles, hence every positive elements $[V]$ in $K^0(\C P^{n-1})^+\smallsetminus\{0\}$ is a linear combination of classes $[\mathcal{L}_k]$, $k\in\Z$, with positive coefficients. It follows from the above discussion that $[V]=a_0+a_1x+a_2x^2+O(x^3)$ has $a_2\geq 0$.
Thus $1-x^2\notin K^0(\C P^{n-1})^+$. Since $1-x=[\mathcal{L}_1]$ is positive and $1-x^2=(1-x)(1+x)$, this means also that $1+x\notin K^0(\C P^{n-1})^+$.
\end{proof}

For $n\geq 3$ it is not an easy task to characterize the positive cone.
For example, there are polynomials with positive coefficients, such as $1+x$, that do not belong to $K^0(\C P^{n-1})^+$.


\begin{thebibliography}{10}

\bibitem{AER22}
S. Arklint, S. Eilers and E. Ruiz, \textit{Geometric classification of isomorphism of unital graph $C^*$-algebras}, New York J. Math. 28 (2022), 927--957. 
\web{www.emis.de/journals/NYJM/nyjm/j/2022/28-38.html}

\bibitem{ADHT22}
F. Arici, F. D'Andrea, P.M. Hajac and M. Tobolski, \textit{An equivariant pullback structure of trimmable graph C*-algebras}, J. Noncommut. Geom. 16(2022), 761--785. \doi{10.4171/JNCG/421}

\bibitem{AT60}
M. Atiyah and J. Todd, \textit{On complex stiefel manifolds}, Math. Proc. Cambridge Phil.
Soc. 56 (1960), 342--353. \doi{10.1017/S0305004100034642}

\bibitem{BPRS00}
T. Bates, D. Pask, I. Raeburn and W. Szyma\'nski, \textit{The C*-algebras of row finite graphs},
New York J.~Math. 6 (2000), 307--324. \web{https://www.emis.de/journals/NYJM/nyjm/nyjm/j/2000/6-14.html}

\bibitem{BHRS02}
T. Bates, J.H. Hong, I. Raeburn and W. Szyma{\'n}ski, \textit{The ideal structure of the C*-algebras of infinite graphs}, Illinois J. Math. 46 (2002) 1159--1176. \doi{10.1215/ijm/1258138472}

\bibitem{BBKS22}
T. Brzezi{\'n}ski,  R. {\'O} Buachalla, U. Kr{\"a}hmer and K. Strung, in preparation.

\bibitem{BDCH17}
P.F. Baum, K.~De Commer, and P.M. Hajac, \emph{Free actions of compact quantum groups on unital C*-algebras}, Doc. Math. 22 (2017), 825--849.
  
\bibitem{BS18}
T. Brzezi{\'n}ski and W. Szyma{\'n}ski, \textit{The C*-algebras of quantum lens and weighted projective spaces}, J. Noncommut. Geom. 12 (2018) 195--215.
\doi{10.4171/JNCG/274}

\bibitem{CHR19}
L.O. Clark, R. Hazrat and S.W. Rigby, \textit{Strongly graded groupoids and strongly graded Steinberg algebras}, J. Algebra 530 (2019) 34--68. \doi{10.1016/j.jalgebra.2019.03.030}

\bibitem{Dan23}
F. D'Andrea, \textit{Quantum spheres as graph C*-algebras: a review}, Expositiones Mathematicae 42 (2024), 125632. \doi{10.1016/j.exmath.2024.125632}

\bibitem{Dan24}
F. D'Andrea, \textit{On the K-theory of the AF core of a graph C*-algebra}, \arxiv{2410.06242}.

\bibitem{DHMSZ20}
F. D'Andrea, P.M. Hajac, T. Maszczyk, A. Sheu and B. Zielinski, \textit{The K-theory type of quantum CW-complexes}, \arxiv{2002.09015}.

\bibitem{DT02}
D. Drinen and M. Tomforde, \textit{Computing K-theory and Ext for graph C*-algebras},
Illinois J. Math. 46 (2002), 81--91. \arxiv{math/0009228}

\bibitem{DT05}
D. Drinen and M. Tomforde, \textit{The C*-algebras of arbitrary graphs}, Rocky Mountain J. Math. 35 (2005) 105--135. \doi{10.1216/rmjm/1181069770}

\bibitem{ERS12}
S. Eilers, E. Ruiz, A.P.W. S{\o}rensen, \textit{Amplified graph C*-algebras}, M{\"u}nster J. of Math. 5 (2012), 121--150. \doi{10.48550/arXiv.1110.2758}

\bibitem{Ell76}
G.A. Elliott, \textit{On the classification of inductive limits of sequences of
semisimple finite-dimensional algebras}, J. Algebra 38 (1976), 29--44. \doi{10.1016/0021-8693(76)90242-8}

\bibitem{Ell00}
D.A. Ellwood, \emph{A new characterisation of principal actions}, Journal of
  Functional Analysis 173 (2000), 49--60.

\bibitem{ER19}
S. Eilers and E. Ruiz, \textit{Refined moves for structure-preserving isomorphism of graph $C^*$-algebras}, \arxiv{1908.03714}.

\bibitem{FLR00}
N.J. Fowler, M. Laca and I. Raeburn, \textit{The C*-algebras of infinite graphs}, Proc. Amer.
Math. Soc. 128 (2000) 2319--2327. \doi{10.1090/S0002-9939-99-05378-2}

\bibitem{Haz13}
R. Hazrat, \textit{The graded structure of Leavitt path algebras}, Israel J. Math. 195 (2013) 833--895. \doi{10.1007/s11856-012-0138-5}

\bibitem{hnpsz18}
P.M. Hajac, R. Nest, D. Pask, A. Sims and B. Zieli{\'n}ski, 
\textit{The K-theory of twisted multipullback quantum odd spheres and complex projective spaces}, J. Noncommut. Geom. 12 (2018), 823--863. \doi{10.4171/JNCG/292}

\bibitem{HRT18}
P.M. Hajac, S. Reznikoff and M. Tobolski, \textit{Pullbacks of graph C*-algebras from admissible pushouts of graphs},
Banach Center Publ. 120 (2020) 169--178. \doi{10.4064/bc120-13}

\bibitem{HS02}
J.H. Hong and W. Szyma{\'n}ski, \textit{Quantum spheres and projective spaces
 as graph algebras}, Commun. Math. Phys. 232 (2002), 157--188. \doi{10.1007/s00220-002-0732-1}

\bibitem{HS03}
J.H. Hong and W. Szyma{\'n}ski, \textit{Quantum lens spaces and graph algebras}, Pacific J. Math. 211 (2003) 249--263. \doi{10.2140/pjm.2003.211.249}

\bibitem{HS08}
J.H. Hong and W. Szyma{\'n}ski, \textit{Noncommutative balls and mirror quantum spheres}, J. London Math. Soc. 77 (2008), 607--626. \doi{10.1112/jlms/jdn003}

\bibitem{KPR98}
A. Kumjian, D. Pask and I. Raeburn, \textit{Cuntz-Krieger algebras of directed graphs}, Pacific J. Math. 184 (1998), 161--174. \doi{10.2140/pjm.1998.184.161}

\bibitem{LO22}
P. Lundstr{\"o}m and J. {\"O}inert, \textit{Strongly graded Leavitt path algebras}, J. Algebra App. 21 (2022) 2250141. \doi{10.1142/S0219498822501419}

\bibitem{Mon93}
S. Montgomery, \textit{Hopf algebras and their actions on rings}, RCSM no. 82, Amer. Math. Soc., Providence, RI, 1993.

\bibitem{NT12}
S. Neshveyev and L. Tuset, \textit{Quantized algebras of functions on homogeneous spaces with Poisson stabilizers}, Commun. Math. Phys. 312 (2012), 223--250. \doi{10.1007/s00220-012-1455-6}

\bibitem{P99}
G.K. Pedersen, \textit{Pullback and Pushout Constructions in C*-Algebra Theory},
J. Funct. Anal. 167 (1999), 243--344. \doi{10.1006/jfan.1999.3456}

\bibitem{R05}
I.~Raeburn, \textit{Graph algebras}, CBMS Regional Conference Series in Mathematics 103, AMS 2005. \isbn{9780821836606}

\bibitem{Rud12}
J. Rudnik, \textit{The K-theory of the triple-Toeplitz deformation of the complex projective plane}, Banach Center Pub. 98 (2012) 303--310. \doi{10.4064/bc98-0-13}

\bibitem{Sch90}
H.-J. Schneider, \textit{Principal homogeneous spaces for arbitrary Hopf algebras}, Israel J. Math. 72 (1990), 167--195. \doi{10.1007/BF02764619}

\bibitem{S22}
K. Strung, \textit{Realizing quantum flag manifolds as graph C*-algebras},
Oberwolfach Report 36 (2022), 2066--2069.
\doi{10.4171/OWR/2022/36}

\bibitem{Swe69}
M. Sweedler, \textit{Hopf Algebras}, Math. lecture note series, W.A. Benjamin, 1969.

\bibitem{VS90}
L.L. Vaksman and Y.S. Soibelman, \textit{The algebra of functions on quantum $SU(n + 1)$ group and odd-dimensional quantum spheres}, Leningrad Math. J. 2 (1991) 1023--1042; translation from Algebra i Analiz 2 (1990) 101--120. \web{http://mi.mathnet.ru/aa208}

\end{thebibliography}
\end{document}